\documentclass[preprint,12pt]{elsarticle}
\makeatletter
\def\ps@pprintTitle{%
  \let\@oddhead\@empty
  \let\@evenhead\@empty
  \let\@oddfoot\@empty
  \let\@evenfoot\@oddfoot
}
\makeatother

\usepackage{amssymb}
\usepackage{amsmath}
\usepackage{amsthm}
\usepackage{url}
\newcommand{\comm}[1]{}

\newtheorem{teorema}{Theorem}[section]

\newtheorem{lema}[teorema]{Lemma}
\newtheorem{proposicion}[teorema]{Proposition}
\newtheorem{corolario}[teorema]{Corollary}

\theoremstyle{definition}

\newtheorem{procedimiento}[teorema]{Procedure}

\newtheorem{definicion}[teorema]{Definition}

\newtheorem{warning}[teorema]{Warning}
\newtheorem{advice}[teorema]{Advice}

\newtheorem{convencion}[teorema]{Convention}

\newtheorem{construccion}[teorema]{Construction}

\newtheorem{experimento}[teorema]{Experiment}

\newtheorem{ejercicio}[teorema]{Exercise}

\newtheorem{observacion}[teorema]{Observation}
\newtheorem{remark}[teorema]{Remark}
\newtheorem{computacion}[teorema]{Computation}

\newtheorem{notacion}[teorema]{Notation}

\DeclareMathOperator{\rcs}{rcs}
\DeclareMathOperator{\trun}{trun}
\DeclareMathOperator{\sym}{sym}

\DeclareMathOperator{\suma}{sum}
\DeclareMathOperator{\un}{un}
\DeclareMathOperator{\mult}{mult}
\DeclareMathOperator{\pol}{pol}

\begin{document}

\begin{frontmatter}



\title{Guessing sequences of eigenvectors for LMPs defining spectrahedral relaxations of Eulerian rigidly convex sets}


\author{Alejandro González Nevado} %% Author name

\affiliation{organization={Universität Konstanz, Fachbereich Mathematik und Statistik},%Department and Organization
            addressline={Universitätstraße 10}, 
            city={Konstanz},
            postcode={78464}, 
            state={Baden-Württemberg},
            country={Germany}}

%% Abstract
\begin{abstract}
Stable multivariate Eulerian polynomials were introduced by Brändén (and first discussed in \cite{branden2011proof,visontai2013stable, haglund2012stable}). Particularizing some variables, \cite[Corollary 6.9]{ale1} shows that it is possible to extract real zero multivariate Eulerian polynomials from them. These real zero multivariate Eulerian polynomials can be fed into constructions of  spectrahedral relaxations providing therefore approximations to the (Eulerian) rigidly convex sets defined by these polynomials. The accuracy of these approximations is measured through the behaviour in the diagonal, where the usual univariate Eulerian polynomials sit. In particular, in this sense, in \cite{ale1} is shown that the accuracy of the global spectrahedral approximation produced by the relaxation can be measured in terms of bounds for the extreme roots of univariate Eulerian polynomials. The bounds thus obtained beat the previous bounds found in the literature. However, the bound explicitly studied and obtained in \cite{ale1} beat the previously known bounds by a quantity going to $0$ when $n$ goes to infinity. Here we use numerical experiments to construct a sequence of vectors providing a (linearized) bound whose difference with the previous known bounds is a growing exponential function (going fast to infinity when $n$ grows). This allows us to establish a better (diagonal) measure of accuracy for the spectrahedral relaxation of the Eulerian rigidly convex sets provided in \cite{ale1}. In particular, we will achieve this linearizing through the sequence of vectors $\{(y,(-2^{m-i})_{i=3}^{m},(0,\frac{1}{2}),(1)_{i=1}^{m})\in\mathbb{R}^{n+1}\}_{n=1}^{\infty}$ for $n=2m$.
\end{abstract}


%% Keywords

\end{frontmatter}




\section{Introduction}\label{intro}
We will present here an analogy between cases of sequences of polynomials showing us the importance of the study of the sequence of Eulerian polynomials (dating all the way back to \cite{euler}) from the point of view of real algebraic geometry (which can be traced back to Fourier's algorithm in \cite{fourier1826solution}). The fact that Eulerian polynomials were essentially existing in the literature 90 years before Fourier published his algorithm shows us that these polynomials have mathematical impact that goes beyond just the nature of their zeros.

However, precisely because of the importance of these polynomials in many other aspects and theories, the interest in the nature, computation and behaviour of their roots is already justified and certainly currently present in the literature, as one can see easily checking, e.g., \cite{sobolev2006selected, savage2013eulerian, savage2015s, yang2015real, mezo2019combinatorics, haglund2019real, alexandersson2025real}. Now, we want to also build a highway driving in the opposite direction: Eulerian polynomials are, due to their relevance in these many mentioned areas, the best laboratory to test techniques in RAG (real algebraic geometry) and even to discover and test new venues and patterns susceptible to be of independent interest to other fields related to RAG and beyond. This is a phenomenon that is getting more prominence and research nowadays in several different forms, as can be seen in \cite{katz2024matroidal,branden2011proof,haglund2012stable,visontai2013stable}.

In particular, a deeper study of these polynomials can serve the purpose of testing the power of some of our techniques (as the one in \cite{main} we mainly focus on here) to estimate roots (see \cite{sobolev2006selected,mezo2019combinatorics}) and to understand phenomena about several notions of polynomial lifting and injection (see \cite{haglund2012stable,visontai2013stable}). This study has therefore the power and ambition to expand our view about the horizons of modern polynomial research. It does so by slowly shaping a path along new landscapes that such research provides. It also shows, at the same time, many new openings and ways shaped and enlightened through the connections provided by the relations that these objects establish between the different realms where they root and emerge from. Thus, this analysis offers us a wide prospect about the huge panorama bridged and intertwined together in this manner through the understanding of the structures hidden between the roots and coefficients of these polynomials and their many liftings.

In order to start our introduction to these polynomials here in analogy to an even more classical sequence of polynomials, we will need to be able to transform sequences into polynomials. This transformation that we will use to construct these polynomials here will show us that the extension of the span of the mentioned conceptual roots of the Eulerian polynomials goes even much deeper than we already wrote into the turbulent ground of mathematics.

\begin{definicion}[Polynomialization]
Let $n\in\mathbb{N}$ and consider a finite sequence $s\colon \{0\}\cup[n-1]\to\mathbb{K}$ for a field $\mathbb{K}$. We define the \textit{polynomialization of the sequence} $s$ as $\pol(s,x):=\sum_{i=0}^{n-1}s(i)x^{i}\in\mathbb{K}[x]$.
\end{definicion}

A polynomialization is therefore just a finite generating series. When we are given numbers in a bidimensional array (like a matrix or a triangle), we can transform the information of such arrays into polynomial sequences through polynomialization. In classical mathematics, the most prominent triangles are those for which each row is symmetric through the main altitude and begins and ends with $1$. These triangles are called in \cite{barry2006integer} \textit{generalized Pascal triangles}. The two more famous examples of these triangles are Pascal's triangle itself and Euler's triangle. We will want our polynomialization to be real rooted and, in a further step, we will want analogues of real rootedness in multiple variables. The analogous family that we are going to use here is the family of RZ polynomials.

\begin{definicion}[Real zero polynomial]\cite[Definition 2.1]{main}\label{RZ}
Let $p\in\mathbb{R}[\mathbf{x}]$ be a polynomial. We say that $p$ is a \textit{real zero polynomial} (or a \textit{RZ polynomial} for short) if, for all directions $a\in\mathbb{R}^{n}$, the univariate restriction $p(ta)$ verifies that all its roots are real, i.e., if, for all $\lambda\in\mathbb{C}$, we have $$p(\lambda a)=0 \Rightarrow \lambda\in\mathbb{R}.$$
\end{definicion}

It is immediate that a RZ polynomial $p$ verifies $p(0)\neq0$. There are however extensions of this definition that allow $p(0)=0$, see, e.g., \cite{knese2016determinantal}. Thanks to the fact that $p(0)\neq0$ it is possible to introduce the following subset of the domain $\mathbb{R}^{n}$ of the polynomial $p$. This subset is of \textit{central} importance to our research.

\begin{definicion}[Rigidly convex set]\cite[Definition 2.11]{main}\label{RCS}
Let $p\in\mathbb{R}[\mathbf{x}]$ be a RZ polynomial and call $V\subseteq\mathbb{R}^{n}$ its set of zeros. We call the \textit{rigidly convex set} of $p$ (or its \textit{RCS} for short) the Euclidean closure of the connected component of the origin $\mathbf{0}$ in the set $\mathbb{R}^{n}\smallsetminus V$ and denote it $\rcs(p)$.
\end{definicion}

Using wisely the famous triangles of numbers we mentioned above we can build examples of RZ polynomials. The construction with Pascal's triangle is well-known and immediate.

\begin{construccion}[Pascal's triangle]
Polynomializing the $i$-th row of Pascal's triangle gives the polynomial $(1+x)^{i}$. Thus the Pascal triangle encodes the sequence of polynomials $\{(1+x)^{i}\}_{i=0}^{\infty}$. This sequence clearly admits the following simple recurrence relation $A_{n}=(1+x)A_{n-1}.$ We can therefore easily lift that sequence of real rooted polynomials to a multivariate sequence of RZ polynomials\footnote{These polynomials are so simple that we do not even need to use Borcea-Br\"and\'en theory of stability preserving operators here.} introducing a new variable in each iteration of the recurrence forming thus the sequence of multivariate polynomials $\{\prod_{j=1}^{i}(1+x_{j})\}_{i=0}^{\infty}.$ We call this sequence the \textit{Pascal multivariate sequence}.
\end{construccion}

It is clear that the polynomials in the Pascal multivariate sequence verify that, for each $i$, the rigidly convex set $C_{i}$ of the $i$-th polynomial is the corresponding positive orthant $O_{i}$ of $\mathbb{R}^{i}$ up to a translation by $\mathbf{1}\in\mathbb{R}^{i}$, that is, $C_{i}+\mathbf{1}=O_{i}$. Thus we see that the corresponding RCSs $C_{i}$ in the sequence $\{C_{i}\}_{i=0}^{\infty}$ is, not only spectrahedral, but actually polyhedral. It is therefore natural to ask what happens for the corresponding construction following the other well-known triangle. Sadly, we will see that not much is known so far.

\begin{construccion}[Euler's triangle]
Polynomializing the $i$-th row of Euler's triangle gives the (univariate) Eulerian polynomial $A_{i}$ of degree $i$. Thus we have just built the sequence $A=\{A_{i}\}_{i=0}^{\infty}$ of Eulerian polynomials. It is well known that this sequence admits the following simple recurrence relation $$A_{n}=(n+1)xA_{n-1}(x)+(1-x)(xA_{n-1})'$$ for $n>0$. Now, we can see that this recurrence involves taking a derivative of the previous element in the sequence $A_{n-1}$ with respect to the external variable $x$ of the operator (or considering the sequence of derivatives of Eulerian polynomials, which would not be more helpful in this situation). This implies that we will have to do further work to lift this sequence of univariate real rooted polynomials to a sequence of multivariate RZ polynomials. We can indeed also lift that sequence of real rooted polynomials to a multivariate sequence of RZ polynomials\footnote{Here using Borcea-Br\"and\'en theory of stability preserving operators will be indeed necessary to proceed.} but this will require that we introduce an additional set of variables. In particular, we can proceed homogenizing the sequence (through the new variable $y$) in order to obtain the sequence of homogenenizations $A^{h}=\{A^{h}_{i}\}_{i=0}^{\infty}.$ This sequence of homogeneous polynomials verifies therefore, for $n>0$, the recursion \begin{gather*}A^{h}_{n}=(n+1)xA^{h}_{n-1}+(y-x)\frac{\partial}{\partial x}(xA^{h}_{n-1})=\\(n+1)xA^{h}_{n-1}+yA^{h}_{n-1}+yx\frac{\partial}{\partial x}A^{h}_{n-1}-xA^{h}_{n-1}-x^{2}\frac{\partial}{\partial x}A^{h}_{n-1}=\\nxA^{h}_{n-1}+yA^{h}_{n-1}+yx\frac{\partial}{\partial x}A^{h}_{n-1}-x^{2}\frac{\partial}{\partial x}A^{h}_{n-1}=\\(nx+y)A^{h}_{n-1}+(yx-x^{2})\frac{\partial}{\partial x}A^{h}_{n-1}.\end{gather*} We will now immediately see that we can rewrite this recursion in a more revealing form as \begin{gather}\label{friendlier} A^{h}_{n}=(x+y)A^{h}_{n-1}+xy\left(\frac{\partial}{\partial x}+\frac{\partial}{\partial y}\right)A^{h}_{n-1}.\end{gather} The proof of this is simple. We can see this fact following the differing terms in order to verify that \begin{gather*}
    (n-1)xA_{n-1}^{h}-x^{2}\frac{\partial}{\partial x}A^{h}_{n-1}=xy\frac{\partial}{\partial y}A^{h}_{n-1},
\end{gather*} which is equivalent to \begin{gather*}
    (n-1)A_{n-1}^{h}-x\frac{\partial}{\partial x}A^{h}_{n-1}=y\frac{\partial}{\partial y}A^{h}_{n-1},
\end{gather*} which is also equivalent to \begin{gather*}
    (n-1)A_{n-1}^{h}=(x\frac{\partial}{\partial x}+y\frac{\partial}{\partial y})A^{h}_{n-1},
\end{gather*} which is clear noticing that $A^{h}_{n-1}$ has degree $n-1$ and noticing what the operator $x\frac{\partial}{\partial x}+y\frac{\partial}{\partial y}$ actually does because \begin{gather*}
    \left(x\frac{\partial}{\partial x}+y\frac{\partial}{\partial y}\right)x^{n}y^{m}=nx^{n}y^{m}+mx^{n}y^{m}=(n+m)x^{n}y^{m}.
\end{gather*} This finishes the proof for the recursion of the homogeneous polynomials. We can therefore now use the expression of the homogenized recursion in Equation \ref{friendlier} to lift to a multivariate sequence. We do this introducing a new pair of variables in each iteration. In particular, we force new variables in each iteration by setting, for $n>0$, $A^{h}_{n}=$  \begin{gather*}(x_{n}+y_{n})A^{h}_{n-1}+x_{n}y_{n}\sum_{i=1}^{n-1}\left(\frac{\partial}{\partial x_{i}}+\frac{\partial}{\partial y_{i}}\right)A^{h}_{n-1}.\end{gather*} Now Borcea-Br\"and\'en theory of stability preserving operators ensures that the polynomials in this sequence of multivariate homogenenous polynomials are real stable \cite[Subsection 2.5]{haglund2012stable}. We can also ensure, as we did in \cite{ale1}, that the sequence of dehomogenizations $A:=\{A^{h}_{i}|_{\mathbf{y}=\mathbf{1}}\}_{i=0}^{\infty}$ is formed by RZ polynomials verifying that the $i$-th polynomial has $i$ different variables $(x_{1},\dots,x_{i})$ and degree $i$.
\end{construccion}

This is how we can synthetically build our lifting of Eulerian polynomials using stability preserving operators and lifting in the recursion. All this without looking at the combinatorial meaning of such lifting or the recursion. These multivariate Eulerian polynomials of course have combinatorial meaning, as we can see in \cite{haglund2012stable, visontai2013stable, ale1}. Observe that this synthetic way of introducing the sequence of RZ multivariate Eulerian polynomials makes us forget about the important underlying combinatorial structure but, at the same time, it allows us to avoid the appearance of ghost variables that arise naturally when we construct these polynomials through the descent top sets of permutations on the ordered set $[n]$ (see \cite{ale1}) because $1$ can never be a top. However, as we are able to build them without looking at the underlying combinatorial structure we do not need to go deeper into these problems and features here.

What interests us the most here is how the process described in the construction above gives us, in a natural way, a multivariate analogue to the sequence of univariate Eulerian polynomials, that is, to the sequence of polynomials constructed polynomializing the rows of the second most famous generalized Pascal triangle. This is remarkable because the process in the Pascal triangle gives RZ polynomials whose RCSs are, up to a simple translation by the vector $\mathbf{1}$ of all ones, the positive orthants in all dimensions. Thus, the Eulerian RCSs are a natural object to look at and put in comparison with the positive orthants. However, while a lot is known about the positive orthants, since these are very simple sets, not so much is known about Eulerian RCSs. In particular, it is clear that the positive orthants are polyhedra (thus also spectrahedra). Meanwhile, we do not even know if the Eulerian RCSs are spectrahedra. This is why in this paper (following the trail of \cite{ale1}) we research the accuracy of spectrahedral approximations (in fact, relaxations \cite[Definition 4.6]{ale1}) to these Eulerian RCSs.

In order to measure the accuracy of these approximations we look at their behaviour in the diagonal, where we find the well-known univariate version of the Eulerian polynomials. There, we use the accuracy of bounds to the extreme roots in order to measure how good is the relaxation (around the diagonal). In \cite{ale1} we could prove, through a linearization, that the act of increasing the variables does indeed improve the (asymptotic) behaviour of the relaxation (around the diagonal) providing better bounds and, therefore, more accuracy of the global spectrahedral approximation (relaxation) around the diagonal. Here, we will see that these bounds can be strengthened even more using a better sequence of vectors in the linearization. As was noted in \cite{ale1}, these linearizations are necessary in order to avoid the impracticable task of dealing directly with the determinantal polynomials produced by considering the determinant of the LMP of the relaxation. Dealing with that determinant, would give us exactly the bound obtained via the relaxation and, therefore, the accuracy obtained around the diagonal for the global spectrahedral relaxation, but all this would happen at the cost of having to deal with a polynomial even more difficult (and, fundamentally, with less nice symmetries) than the original multivariate Eulerian polynomials defining the RCSs we are interested in, which would be nonsense. This is why we have to find good sequences of vectors to feed the linearization process with.

We explain in the next section the contents of this article. In particular, we stress here that in this article we construct a guess for a nice sequence of vectors fulfilling the promised improvement in our measure of accuracy of the spectrahedral relaxation as an approximation. We do this measurement focusing our attention and tools around the diagonal. We notice that this way of measuring accuracy along (around) the diagonal was proposed as a first rudimentary method to asses this problem on these global approximations to RCSs in \cite{ale1}. In fact, this approach (around the diagonal) towards detecting accuracy of the approximation given by the spectrahedral relaxation provides in particular a measure of the accuracy of our estimation.

It is evident that there could be more and better measures, but the one we use here and in \cite{ale1} fits perfectly for a first approach to the topic. Precisely, this estimation is the tool that allow us to assess how good is the spectrahedral relaxation introduced in \cite{main} as a global approximation to the Eulerian RCSs introduced in \cite{ale1} and discussed above. Here therefore we introduce and explain the process that we will use to construct nice guesses for sequences of approximated eigenvectors for the LMP produced by the spectrahedral relaxation introduced in \cite{main} when restricted to the diagonal. And we do this since we measure accuracy using the behaviour of our approximations and estimations around the diagonal, as stipulated in \cite{ale1} as a rudimentary method for dealing with these problems in an initial approach (at least until these explorations gain more traction).

\section{Contents}

We began the introduction in Section \ref{intro} above constructing the RZ multivariate Eulerian polynomials in a synthetic way. Thus we avoid having to resort to the combinatorial interpretation of these polynomials and gave instead a construction based on generalized Pascal triangles and stability preserving operators lifting univariate recurrence relations to multivariate ones. In the next Section \ref{over} we will give a short overview of the relevance of these polynomials and their multivariate liftings and why these are important in many areas of mathematics. Then in Section \ref{prelim} we provide the reader with the preliminaries to understand the nature of the tool given by the relaxation, which we use here to bound and (laterally) approximate RCSs. We shortly remind in Section \ref{mult} the importance of the fact that the polynomials we are building here are RZ and the impact that that family of polynomials has in the current development of optimization techniques. In Section \ref{count} we give an overview of the counting formulas and strategies used when dealing with these polynomials and compute the values and functions necessary for the construction of their corresponding spectrahedral relaxations \cite{main,ale1}. With the relaxation constructed, we explain in Section \ref{extr} how we extract bounds for the extreme roots from the relaxation through linearization. At this point, in Section \ref{asymp}, we use a short digression to talk about the asymptotic tools we will need to continue our study. In Section \ref{pre} we check and discuss the weakness and benefits of bounds obtained previously using different methods or linearizations than the ones we use here. This discussion let us explore the possibility of constructing better linearizations through the use of better guessing strategies for the approximated eigenvectors that we can build. We do this analysis in Section \ref{guess} taking care of our choice to ensure that we are effectively using the advantages provided by multivariability. The first development of the strength of our new guess is performed in Section \ref{limit}, where we discuss the limitations of our methods and prepare the setting for the computations in the next section. In that Section \ref{form} we provide the form of the bound obtained and, taking care in Section \ref{manag} that we manage the radicals appearing in the optimal $y$ we have to choose correctly, we compute the actual bound, through a simple optimization process, in Section \ref{explo}, where we also see how that bound provides an exponential improvement with respect to the bounds previously discussed in \cite{ale1}. Finally, in Section \ref{other} we discuss other possible paths to improve our bounds and explain the nature of these improvements in Section \ref{nature}. We end our journey with Section \ref{con}, where we conclude this article with an overview of what we have done here in direct connection to what we still need to do in the future, thus providing promising venues and ideas for new and future research topics and possibilities within the framework established, introduced and briefly discussed in this paper.

\section{Overview}\label{over}

The main result of this article will be a measure of accuracy around the diagonal of the spectrahedral approximation provided by the relaxation in terms of bounds given for the extreme roots of the univariate Eulerian polynomials sitting in such diagonal. Reading \cite{ale1} we can see that the jump from the univariate Eulerian polynomials to the multivariate ones makes the bound for the extreme roots provided by the relaxation improve. However, since we cannot effectively directly deal with the LMP provided by the relaxation through the determinant, we had to obtain a certificate of improvement through a strategy of linearization through an \textit{adequate} sequence of vectors. Correctly constructing these vectors in a way that their entries are easy to compute and they are effective for our calculation was a challenging task. It was possible to certificate the improvement when increasing the number of variables in the Eulerian polynomials (and therefore in the relaxation) in \cite{ale1} through a linearization constructed there using the sequence of vectors $\{(y,1,\mathbf{-1})\in\mathbb{R}^{n+1}\}$. This was the first time that the gain in effectivity of the multivariate relaxation over the univariate relaxation through the increase of the number of variables could be proven. However, the sequence of vectors used for this linearization in \cite{ale1} was not entirely satisfactory because of the following three seasons.

\begin{remark}[Plenty of room for improvement]\label{three}
Using the sequence of vectors $\{(y,1,\mathbf{-1})\in\mathbb{R}^{n+1}\}$ for the linearization in \cite{ale1} leave in fact a lot of room for improvement as can be seen in different stages of the process of estimating the accuracy of the bounds obtained. We comment on these.
\begin{enumerate}
    \item \textbf{Experiments point further.} When we do numerical experiments and compute things numerically, it seems like the multivariate relaxation should offer a much greater improvement than the one we see after linearizing. This does not prove anything in principle because it could be that the improvement is only as good as we observed for relatively small indices $n$, i.e., for these indices for which we can make experiments before computations become unmanageably lengthy and time-consuming. However, it is not clear what phenomena could actually hinder the generalization of that greater improvement for higher indices $n$.
    \item \textbf{Vector too simple.} Our linearization mix back up a lot of information about the permutations that the multivariate Eulerian polynomials split. The fact that the tail is constantly $-1$ along the whole sequence does not allow to put something differential in each piece of the partitioned information. This implies that much of the information of the partition performed by the multivariate lifting of the polynomial is irremediably lost during the linearization. This linearization therefore kills much of the additional information that the multivariate relaxation could be actually retrieving. We need sequences of vectors with more structure both in terms of the term of the sequence we are in at each time and in terms of the entry we are looking at each moment. For that, we need that entries in the vectors vary when the index in the sequence varies and that these entries inside each vector show more variety. We have to accomplish this at the same time that we manage that the vector that we build with these features is easy to deal with in symbolic computational terms and, moreover, we want it to be a good guess in the sense that it must resemble enough an actual generalized eigenvalue so that we get a good approximation to the corresponding bonding eigenvalue. There are therefore many different features to balance. And these features are indeed of very different nature.
    \item \textbf{Improvement vanishes at infinity.} After we compute the actual bound and compare it with the univariate bound in \cite{ale1}, we see that it is not just that the certified improvement through the sequence of vectors $\{(y,1,\mathbf{-1})\in\mathbb{R}^{n+1}\}$ is not as great as numeric experiments show for small numbers, but that indeed that improvement \textit{vanishes} at infinity. This contradicts again what we see for small indices when we do experiments and there are not evident reasons why that improvement should vanish. This means that the linearization we are making is probably doing an asymptotic jump that is too big and bad and that therefore kills much of the actual improvement produced by the additional split information collected in the multivariate relaxation.
\end{enumerate}
These three reasons point out in the direction of further exploring the linearizations we are using so that we can find sequences of vectors with more structure and therefore able to carry more information about the multivariate split. In this way, we see that the three points above could be signaling to a clear path for improvement of our management of the linearization process of the relaxation.
\end{remark}

Our work in this article will consist therefore in addressing the three points in the Remark \ref{three} above in a constructively decisive way so that we can actually build a better sequence of vectors providing a certificate of the bounds obtained that better resembles the actual bound given by the relaxation around the diagonal. The initial procedure for accomplishing that improvement does not differ from the procedure followed in \cite{ale1}. This is why, although for completeness we will include it here, we will go through these first steps fast and we will skip these initial proofs here since they can be easily consulted in \cite{ale1} when doing that is needed.

\section{Preliminaries}\label{prelim}

We use the spectrahedral relaxation introduced in \cite{main}. In what follows, all the matrices we will consider are real symmetric. The relaxation works for RZ polynomials. This is why we constructed RZ multivariate liftings of the univariate Eulerian polynomials. Being RZ is just a generalization of being real-rooted for multivariate polynomials. For these RZ polynomials we have in \cite{main}, in fact, a construction providing a \textit{spectrahedral relaxation} of their RCSs.

\begin{definicion}[Relaxation of a RCSs]
Given the RCS $0\in C\subseteq \mathbb{R}^{n}$, we call $\Tilde{C}$ a relaxation of $C$ simply if $C\subseteq\Tilde{C}.$
\end{definicion}

We have more than this. The relaxation we use here is moreover spectrahedral.

\begin{definicion}[Spectrahedra]\cite{ramana1995some}
A set $S\subseteq\mathbb{R}^{n}$ is a \textit{spectrahedron} if it can be written as $$S=\{a\in\mathbb{R}^{n}\mid A_{0}+\sum_{i=1}^{n}a_{i}A_{i} \mbox{\ is PSD\ }\}$$ with $A_{i}$ real symmetric matrices of the same size. We call $A_{0}+\sum_{i=1}^{n}x_{i}A_{i}$ a \textit{linear matrix polynomial} (or LMP) and the condition $$a\in\mathbb{R}^{n} \mbox{\ with \ } A_{0}+\sum_{i=1}^{n}a_{i}A_{i} \mbox{\ is PSD\ }$$ a \textit{linear matrix inequality} (or LMI) on $a\in\mathbb{R}^{n}$.
\end{definicion}

The main points for the construction of this spectrahedral relaxation can be consulted in \cite{main} in general. For a more particular treatment for the case of Eulerian polynomials, the reader can consult \cite{ale1}. We present now the main tools necessary to begin to talk about the construction of the relaxation. First, it is clear that, for certain power series, it makes sense to consider their logarithm.

\begin{definicion}[Logarithm for power series]\cite[Definition 3.1, Definition 3.2]{main}
Let $p=\sum_{\alpha\in\mathbb{N}_{0}^{n}}a_{\alpha}\mathbf{x}^{\alpha}\in\mathbb{R}[[\mathbf{x}]]$ be a real power series with $p(0)=1$. We define $$\log{p}:=\sum_{k=1}^{\infty}(-1)^{k+1}\frac{(p-1)^{k}}{k}\in\mathbb{R}[[\mathbf{x}]]$$ the \textit{logarithm} of $p$.
\end{definicion}

This logarithm immediately let us construct the $L$-forms. These are the main tools to fill the entries of the matrices of the LMPs defining the spectrahedral relaxation in \cite{main}.

\begin{definicion}[$L$-form]\cite[Proposition 3.4]{main}
Let $p\in\mathbb{R}[\mathbf{x}]$ be a polynomial with $p(0)\neq0$. We define the linear form $L_{p}$ on $\mathbb{R}[\mathbf{x}]$ by specifying it on the monomial basis of $\mathbb{R}[\mathbf{x}]$. We set $L_{p}(1)=\deg(p)$ and define implicitly the rest of the values by requiring the identity of the formal power series $$-\log\frac{p(-x)}{p(0)}=\sum_{\alpha\in\mathbb{N}_{0}^{n}\\\alpha\neq0}\frac{1}{|\alpha|}\binom{|\alpha|}{\alpha}L_{p}(\mathbf{x}^{\alpha})\mathbf{x}^{\alpha}\in\mathbb{R}[[\mathbf{x}]].$$
\end{definicion}

These forms, for small degree monomial, have values that are easy to compute. These are precisely the ones we are most interested in for constructing the relaxation.

\begin{computacion}[$L$-form for degree up to three]\cite[Example 3.4]{main}
Suppose $p\in\mathbb{R}[\mathbf{x}]$ verifies $$\trun_{3}(p)=1+\sum_{i\in[n]}a_{i}x_{i}+\sum_{i,j\in[n]\\i\leq j}a_{ij}x_{i}x_{j}+\sum_{i,j,k\in[n]\\i\leq j\leq k}a_{ijk}x_{i}x_{j}x_{k}.$$ Then we have the identities \begin{gather*}L_{p}(x_{i})=a_{i},\\L_{p}(x_{i}^{2})=-2a_{ii}+a_{i}^{2},\\L_{p}(x_{i}^{3})=3a_{iii}-3a_{i}a_{ii}+a_{i}^{3}\end{gather*} for all $i\in[n]$, while \begin{gather*}L_{p}(x_{i}x_{j})=-a_{ij}+a_{i}a_{j},\\L_{p}(x_{i}^{2}x_{j})=a_{iij}+a_{i}a_{ij}-a_{j}a_{ii}+a_{i}^{2}a_{j}\end{gather*} for all $i,j\in[n]$ with $i<j$ and $$L_{p}(x_{i}x_{j}x_{k})=\frac{1}{2}(a_{ijk}-a_{i}a_{jk}-a_{j}a_{ik}-a_{k}a_{ij}+2a_{i}a_{j}a_{k})$$ for all $i,j,k\in[n]$ with $i<j<k$.
\end{computacion}

We stop at degree $3$ because this is all we need to construct the spectrahedral relaxation introduced in \cite{main}. In particular, the spectrahedral relaxation has the form of a LMI whose LMP is PSD at the origin and whose entries are defined in terms of values of the $L$-form on monomials of degree up to three. We use mold matrices to build the spectrahedral relaxation.

\begin{notacion}[Mold matrices and entrywise transformations]
We denote $L\circledcirc M$ the tensor in $B^{s_{1}\times\cdots\times s_{k}}$ obtained after applying the map $L\colon A\to B$ to each entry of $M\in A^{s_{1}\times\cdots\times s_{k}}.$ The tensor $M$ is called the \textit{mold tensor} and the map $L$ is the \textit{molding map}.
\end{notacion}

Now we can build the spectrahedral relaxation using the simple mold matrix given by the symmetric moment matrix indexed entryways by all the monomials in the variables $\mathbf{x}$ up to degree $1$. We name it.

\begin{notacion}
We denote $M_{n,\leq1}:=(1, x_{1}, \cdots, x_{n})^{T}(1, x_{1}, \cdots, x_{n})=$ \begin{equation}\label{mainmoment}
\begin{pmatrix}
1 & x_{1} & \cdots & x_{n}\\
x_{1} & x_{1}^{2} & \cdots & x_{1}x_{n}\\
\vdots & \vdots & \ddots & \vdots\\
x_{n} & x_{1}x_{n} & \cdots & x_{n}^2
\end{pmatrix}
\end{equation} the \textit{mold matrix indexed by monomials of degree up to $1$}.
\end{notacion}

Over this mold we easily build the spectrahedral relaxation. We proceed as follows.

\begin{definicion}[Spectrahedral relaxation]\label{rela}\cite[Definition 3.19]{main}
Let $p\in\mathbb{R}[\mathbf{x}]$ be a polynomial with $p(0)\neq0$ and consider the symmetric matrices $A_{0}=L_{p}\circledcirc M_{n,\leq1}$ and $A_{i}=L_{p}\circledcirc (x_{i}M_{n,\leq1})$ for all $i\in[n]$. We call the linear matrix polynomial $$M_{p}:=A_{0}+\sum_{i=1}^{n}x_{i}A_{i}$$ the \textit{pencil associated to $p$} and $$S(p):=\{a\in\mathbb{R}^{n}\mid M_{p}(a) \mbox{\ is PSD}\}$$ the \textit{spectrahedron associated to $p$}.
\end{definicion}

This construction has indeed the properties we want. In particular, it provides us indeed with a relaxation.

\begin{teorema}[Relaxation]\cite[Theorem 3.35]{main}
Let $p\in\mathbb{R}[\mathbf{x}]$ be a RZ polynomial. Then $\rcs(p)\subseteq S(p).$
\end{teorema}

Now we only have to ensure that our polynomials are RZ since the theorem above guarantees us a spectrahedral relaxation if we feed our construction in Definition \ref{rela} with a RZ polynomial. This is what we see shortly in the next section.

\section{Multivariate Eulerian polynomials are real zero}\label{mult}

This was already explained in some detail in the introduction in Section \ref{intro}, but we will shortly expand the story of these polynomials here. The polynomials introduced in \cite[Theorem 3.2]{haglund2012stable}
and \cite[Theorem 3.3]{visontai2013stable} are real stable. We have to set the variables tagging the ascent tops $\mathbf{y}=\mathbf{1}$ in order to transform these real stable polynomials into RZ polynomials. For that, it is necessary to study the relations between real stability, hyperbolicity and RZ-ness. These are all concepts generalizing to real rootedness to the multivariate setting. The key result to prove this path is the following.

\begin{proposicion}[Real stability iff hyperbolicity in all positive directions]\label{realstab}\cite[Proposition 1.3]{pemantle2012hyperbolicity} and \cite[Section 5]{kummer2015hyperbolic}
A homogeneous polynomial $p\in\mathbb{R}[\mathbf{x}]$ is real stable if and only if, for any direction $e\in\mathbb{R}^{n}_{>0}$ with positive coordinates and $a\in\mathbb{R}^{n}$, the univariate polynomial $p(a-te)$ is real-rooted. 
\end{proposicion}

Proposition \ref{realstab} above connects real stability and hyperbolicity. The following equivalence finishes the connection with RZ polynomials relating hyperbolicity to RZ-ness.

\begin{proposicion}[Hyperbolicity and RZ-ness]\cite[Proposition 6.7]{main}
\label{dehomo}
Let $p\in\mathbb{R}[x_{0},\mathbf{x}]$ be a homogeneous polynomial. Then $p$ is hyperbolic in the direction of the first unit vector $u=(1,\mathbf{0})\in\mathbb{R}^{n+1}$ if and only if its dehomogenization $q=p(1,x_{1},\dots,x_{n})\in\mathbb{R}[\mathbf{x}]$ is a RZ polynomial.
\end{proposicion}

The combination of these two proposition allows us to conclude immediately what we want. Our polynomials here fit the setting of the relaxation.

\begin{corolario}[RZ-ness of dehomogenization]
$A_{n}(\mathbf{x},\mathbf{1})$ is RZ.
\end{corolario}

Now we just need to construct the spectrahedral relaxation. For this, we have to count over some permutations and compute, in this way, several $L$-forms.

\section{Counting descent tops, $L$-forms and LMPs}\label{count}

We will need to count permutations with a prescribed descent set. We introduce a notation for the set of these permutations.

\begin{definicion}[Exact descent]
\label{rns}
Fix $S\subseteq[n+1]$. We denote the set of permutations that \textit{descend exactly} at $S$ by $$R(n,S):=\{\sigma\in\mathfrak{S}_{n+1}\mid S=\mathcal{DT}(\sigma)\},$$ where $\mathcal{DT}(\sigma)$ is the descent top set of $\sigma$.
\end{definicion}

The cardinals of these sets are the basic objects that we will need to compute in order to fill the entries of the coefficient matrices forming the LMP that defines the spectrahedral relaxation introduced in \cite{main}. In \cite{ale1}, we can see the following expression for these cardinals.

\begin{corolario}[Cardinal of sets with exact descent top set]\cite{ale1}
\label{coroR}
Fix $s=|X|$ and $\{x_{1}<\cdots<x_{s}\}=X\subseteq[n]$ and $\{y_{1}<\cdots<y_{n-s}\}=Y\subseteq[n]$ with $X\cup Y=[n]$. For subsets $S\subseteq Y$, name the ordered chain of elements obtained through the union $X\cup S=\{x_{S,1}<\cdots<x_{S,s+|S|}\}$. Thus, going through the complement, we have that $|R(n-1,X)|=$
\begin{gather}\label{coroR1}
\sum_{S\subseteq[n]\smallsetminus{X}}(-1)^{|S|}(n-|X\cup S|)!\prod_{i=1}^{s+|S|}(x_{S,i}-i).
\end{gather} Similarly, we can express this number in terms of deletions in the initial set as $|R(n-1,X)|=$
\begin{gather}\label{coroR2}
\sum_{J\subseteq X}(-1)^{|X\smallsetminus J|}\alpha(J)\hat{!},
\end{gather} where, for an ordered set $X=\{x_{1}<\dots<x_{k}\}$, we define $$\alpha(X):=(x_{1}-1,x_{2}-x_{1},x_{3}-x_{2},\dots,x_{k}-x_{k-1})$$ and, for a tuple $\beta=(\beta_{1},\dots,\beta_{k})$, we use the operator $$\beta\hat{!}:=(k+1)^{\beta_{1}}k^{\beta_{2}}\cdots3^{\beta_{k-1}}2^{\beta_{k}}.$$
\end{corolario}

As in this article we are not really interested in the sets $R(n,X)$ themselves because we only need their cardinalities, we will abuse the notation in what follows and refer to the cardinal $|R(n,X)|$ of $R(n,X)$ simply as $R(n,X)$. Now we can apply the Corollary \ref{coroR} above to the cases $|X|\in\{1,2,3\}$ fixing $n$ so we can shorten further $R(X):=R(n,X)$. In this way, we obtain the following values.

\begin{computacion}\label{rx}
$R(X)$ equals
\begin{enumerate} 
    \item $2^{x_{1}-1}-1$ for $X=\{x_{1}\}$.
    \item $3^{x_{1}-1}2^{x_{2}-x_{1}}-(2^{x_{1}-1}+2^{x_{2}-1})+1$ for $X=\{x_{1}<x_{2}\}$.
    \item $4^{x_{1}-1}3^{x_{2}-x_{1}}2^{x_{3}-x_{2}}-(3^{x_{1}-1}2^{x_{2}-x_{1}}+3^{x_{2}-1}2^{x_{3}-x_{2}}+3^{x_{1}-1}2^{x_{3}-x_{1}})+(2^{x_{1}-1}+2^{x_{2}-1}+2^{x_{3}-1})-1$ for $X=\{x_{1}<x_{2}<x_{3}\}$.
\end{enumerate}
\end{computacion}

Using these values we can finally compute the $L$-forms that we need to build the spectrahedral relaxation. In order to simplify the notation, we fix $n$ from now on and call $p:=A_{n}(\mathbf{x},\mathbf{1})$. This $p$ is already normalized since $A_{n}(\mathbf{0},\mathbf{1})=1.$

\begin{computacion}
Fix $i,j\in\{0\}\cup[n].$ The value at position $ij$ of the LMP produced by the relaxation is $$\left(\sum_{k\in\{0\}\cup[n]}x_{k}L_{p,d}((x_{k}x_{i}x_{j})|_{x_{0}=1})\right)\bigg{|}_{x_{0}=1}.$$ We will then calculate here the values of the $L$-form over monomials $m$ of degree up to three. For the computation of these values we use the formulas obtained in \cite[Example 3.5]{main} together with Computation \ref{rx}. Thus $L_{p}(m)$ equals, for $i,j\in[n]$,
\begin{enumerate}
    \item $n$ when $m=1$.
    \item $2^{i}-1$ when $m=x_{i}$.
    \item $(2^{i}-1)^{2}$ when $m=x_{i}^{2}$.
    \item $2^{i+j} - 2^{j-i} \cdot 3^{i}$ when $m=x_{i}x_{j}$ with $i<j$.
    \item $(2^{i}-1)^{3}$ when $m=x_{i}^{3}$.
    \item $\frac{1}{3}\cdot 2^{-3 + j - i} (-2 + 2^{i+1}) (-4 \cdot 3^{i+1} + 3 \cdot 4^{i+1})$ when $m=x_{i}^{2}x_{j}$ with $i<j$.
    \item $\frac{1}{3}\cdot2^{-3 + i - j} (-2 + 2^{i+1}) (-4 \cdot 3^{j+1} + 3 \cdot 4^{j+1})$ when $m=x_{i}^{2}x_{j}$ with $j<i$.
    \item $2^{i+j+k} - 2^{j+k-i} \cdot 3^{i} - 2^{-2 + (i+1) - (j+1) + (k+1)} \cdot 3^{j} + 2^{-3 + 2 (i+1) - (j+1) + (k+1)} \cdot 3^{j-i}$ when $m=x_{i}x_{j}x_{k}$ with $i<j<k$. 
    \end{enumerate}
\end{computacion}

There is a difference with the computations in \cite{ale1}. It is a good idea to note this to avoid confusions. The difference is related to the way chosen to introduce Eulerian polynomials in each article.

\begin{warning}[Comparison with expressions in \cite{ale1}]
Notice that the appearance of ghost variables in \cite{ale1} forced us to have different formulas than here. This is so because our introduction of Eulerian polynomials here has not been combinatorial and therefore our polynomials do not show ghost variables. Thus we had to translate $+1$ the indices of the formulas in \cite{ale1}. That is the only change we performed and the reason why the formulas in the computation above look slightly different than in \cite{ale1}. The variable $x_{1}$ is now a \textit{completely appearing} variable and $A_{n}$ has $n$ variables $(x_{1},\dots,x_{n})$ \textit{actually appearing}. Indeed, $x_{i}$ here corresponds with $x_{i+1}$ in \cite{ale1} for $i\in[n]$ and $x_{1}$ is the ghost variable in \cite{ale1}. 
\end{warning}

With these values we can now construct the spectrahedral relaxation of the RZ multivariate Eulerian polynomials. The next step is to extract bounds for the extreme roots of the univariate Eulerian polynomials, which sit along the diagonal. This is the topic of the next section.

\section{Extraction of bounds from the spectrahedral relaxation}\label{extr}

The strategy to effectively extract bounds for the univariate Eulerian polynomials from the spectrahedral relaxation uses the fact that these polynomials are in the diagonal of the multivariate RZ polynomials. Additionally, the fact that univariate Eulerian polynomials are palindromic allows us to transform the lower bound for the rightmost root into an upper bound for the leftmost root. Moreover, to extract bounds we have to linearize because it is not practical to deal with the LMP or its determinant. Thus the way we obtain a bound is by guessing sequences of approximating generalized eigenvectors.

\begin{notacion}[LMP obtained]\label{LMP}
The spectrahedral relaxation of the $n$-th multivariate Eulerian polynomial $A_{n}(\mathbf{x})\in\mathbb{R}[\mathbf{x}]$ gives the LMP $$M_{n}(\mathbf{x}):=M_{n,0}+\sum_{j=1}^{n}x_{j}M_{n,j}\in\sym_{n+1}[\mathbf{x}].$$
\end{notacion}

By the properties of the spectrahedral relaxation we know that $M_{n,0}$ is PSD (see \cite{main}). Fixing $x_{i}=x$ for all $i\in[n]$ we go to the diagonal, where the original univariate Eulerian polynomial $A_{n}(x)\in\mathbb{R}[x]$ lies.

\begin{notacion}[LMP in the diagonal]\label{LMPdia}
At the diagonal we have the univariate LMP $$M_{n}(x,\dots,x)=M_{n,0}+x\sum_{i=2}^{n+1}M_{n,i}:=M_{n,0}+xM_{n,\suma}\in\sym_{n+1}[\mathbf{x}].$$ 
\end{notacion}

Looking to this diagonal gives us therefore a classical generalized eigenvalue problem from which it is now easy to extract bounds. We explain now this procedure to obtain many different linear inequalities which consists in performing an approximate analysis of the corresponding kernel.

\begin{procedimiento}[Obtaining inequalities]
\label{procesobound}
A LMP $$p(x)=A+xB\in\sym_{n}(\mathbb{R})[x]$$ with $A$ a PSD matrix verifies $$v^{\top}(A+xB)v=v^{\top}Av+xv^{\top}Bv\geq0$$ for all $v\in\mathbb{R}^{n}$ and $x\in\rcs(\det(p))$. Thus, for any $v\in\mathbb{R}^{n}$, this choice gives us inequalities of the form $a+xb\geq0$ (one for each $v$). We can rearrange these inequalities solving for $x$ into $x\geq\frac{-a}{b}$ whenever $b>0$. This therefore provides a lower bound for the RCS $\rcs(\det(p))$ of $\det(p)$. In short, $x\in\rcs(\det(p))$ must verify $x\geq\frac{-a}{b}.$
\end{procedimiento}

This is the procedure followed in \cite{ale1} to obtain bounds through the use of the linearizing sequence of vectors $\{(y,1,-\mathbf{1})\in\mathbb{R}^{n+1}\}_{n=1}^{\infty}.$ Here we improve on this vector. To verify and confirm such improvement we will recall the best bound obtained in \cite{ale1} so that we can further on establish comparisons with such bound here. In order to work towards establishing these comparisons we will need to use some concepts from asymptotic expansions. We briefly introduce then for clarity in the next short section.

\section{Rudiments of asymptotic expansions}\label{asymp}

In order to express sequences in terms of their asymptotic expansions we need to establish the scale we are going to use to write these expansions. Our scale is the most natural one for the problem we are treating.

\begin{convencion}[Asymptotic scale]\label{expsca}
Our \textit{asymptotic scale}  (see \cite[Subsection 1.1.2]{paris2001asymptotics}) will be the one given by functions of the form $a^{n}$ for $a\in\mathbb{R}_{\geq0}$. Notice the importance of taking $a\geq0$ and the fundamental difference in the limit of these functions when $a>1$ or $a<1$. This is fundamental for the improvement exposed in this article with respect to what is obtained in \cite{ale1}.
\end{convencion}

In order to stay consistent, it is important to stay with one scale once we have chosen it. We fix such scale across the whole text and discussion about these extreme roots of univariate Eulerian polynomials. This will guarantee the interoperability of our arguments along our whole theory of bounds for these polynomials. Thus we will avoid having to repeat work that was already done somewhere else.

\begin{remark}[Always the same scale]
We fix the scale in Convention \ref{expsca} above all the time in the sense that, whenever we compute some term of the growth of a sequence in this scale, we understand that we do not ever change the scale when we extract more terms of the asymptotic expansion of that same sequence. Thus, our choice of the scale used is homogeneous, stable and invariable throughout our whole analysis all the time. This also makes sense since the scale established and convened above is the most natural one for studying and approaching our bounding problem of the extreme roots of the univariate Eulerian polynomials due to the natural expression of the growth of these roots in it.
\end{remark}

Once we have this natural asymptotic scale we need notation to shorten our treatment of limits. This notation is very common in the asymptotics literature and we recall it here.

\begin{notacion}[Comparisons of limits] \cite{khoshnevisan2007probability}\ \label{insteadO}
Let $f,g\colon\mathbb{N}\to\mathbb{R}_{>0}$ be sequences taking positive values. We write $f\sim g$ if $\lim_{n\to\infty}\frac{f(n)}{g(n)}=1$ and $f(n)=o(g(n))$ as $n\to\infty$ if $\lim_{n\to\infty}\frac{f(n)}{g(n)}=0.$
\end{notacion}

Now we can already recall the best bound obtained in \cite{ale1} since, thanks to this short section, we finally have all the initial tools necessary to establish comparisons with such bound here. The recalling part is what we do in the next section since the comparison part will need some further preparation.

\section{Previous bound}\label{pre}

It was proven in \cite{ale1} that the multivariate bound obtained through the linearizing sequence of vectors $\{(y,1,-\mathbf{1})\in\mathbb{R}^{n+1}\}_{n=1}^{\infty}$ following the method described above was just slightly better than the best univariate bound. In particular, such bound was given as follows.

\begin{lema}[Bound using the sequence of vectors $\{(y,1,-\mathbf{1})\in\mathbb{R}^{n+1}\}_{n=1}^{\infty}$]\label{formofboundfirstwinner}
The bound for the leftmost root of the $n$-th univariate Eulerian polynomial obtained through the linearization of its relaxation in the diagonal by the sequence of vectors $\{(y,1,-\mathbf{1})\in\mathbb{R}^{n+1}\}_{n=1}^{\infty}$ is of the form $q_{n}^{(n)}\geq-\frac{D}{N}$ with $D=$ \begin{gather*}
   10 - 2^{2 + n} + 2^{2 + 2n} - 2 \cdot 3^{1 + n} + n + 4y - 2^{1 + n}y + 
 n y + y(4 - 2^{1 + n} + n + ny)
\end{gather*} and $N=$ \begin{gather*}
    -10 + 2^{3 + n} - \frac{1}{3}2^{3 + 2n} - \frac{1}{3}2^{4 + 2n} + \frac{1}{7}2^{4 + 3n} + \frac{1}{7}2^{5 + 3n} +\\ 2 \cdot 3^n - 4 \cdot 3^{1 + n} + 2 \cdot 3^{2 + n} - \frac{1}{5}2^{1 + n}3^{3 + n} - 4^{1 + n} + 4^{2 + n} - \frac{6^{2 + n}}{5} +\\ \frac{8^{1 + n}}{7} - n - 8y - 2^{2 + n}y + 2^{3 + n}y - \frac{1}{3}2^{3 + 2n}y - \frac{1}{3}2^{4 + 2n}y +\\ 4 \cdot 3^{1 + n}y - 2ny - 2y^2 + 2^{1 + n}y^2 - ny^2,
\end{gather*} where $y=\frac{-b-\sqrt{b^2-4ac}}{2a}$ with $N'D-ND':=ay^2+by+c.$
\end{lema}

Calling the bound above $\mult_{v}$ and the best univariate bound obtained from the relaxation applied to the univariate relaxation $\un$, it was proven in \cite{ale1} that the difference of bounds $\mult_{v}-\un\sim\frac{1}{2}\left(\frac{3}{4}\right)^n.$ Thus we see immediately that we actually obtain an improvement.

However this improvement vanishes at infinity. We want to change that using a better sequence of vectors for linearizing.

We pursue this because the numerical experiments that we can perform with small numbers tell us that this improvement should be bigger and, indeed, the improvement should actually increase. And, although the numerics tell this only for the small numbers over which actual numerical computations are doable, there does not seem to be any theoretical obstacle or hidden phenomena that could diminish and weaken the growth of the actual difference of bounds.

This tells us that this vanishing of the improvement is an artifact of a bad choice of our sequence of guesses of generalized eigenvectors in \cite{ale1}. Hence, a more structured and elaborated choice could potentially show us that the actual improvement is far greater by giving us a certificate of improvement that at least does not vanish so quickly at infinity.

In fact, the room for improvement in this certificate is big and we will indeed see that we can get a certificate of the fact that this difference does not only not decrease towards $0$ so fast at infinity, but that it actually grows. More than that, in fact: it grows exponentially.

Finding the correct choice of guesses of generalized eigenvectors to obtain such certificate is the topic of the next section. There we will explain how we manage to do it and the mechanisms and strategies we use in order to establish nice guesses providing better certificates of how big actually is the improvement given by the multivariate relaxation.

\section{Guessing  eigenvectors with care for effective multivariability}\label{guess}

Here the spectrahedral relaxation will finally show that going multivariate is a decisive factor in order to greatly improve the bounds we get. This section will be heavily computational. We will deal here with computations both symbolically and numerically because we will need to perform many experiments in order to isolate a good sequence of approximated generalized eigenvectors. Fortunately, we will see that this is doable. For ease of notation, however, we will only pursue this for even degree, as this simplifies the form of the sequence of vectors we find. This means that parity will play a role due to the form of the sequence of vectors we will deal with. In any case, we observe that the odd case follows similarly changing minimally the structure of the corresponding sequence of guesses of approximated eigenvectors. Moreover, the fact that the roots of the univariate Eulerian polynomials interlace, allow us to say already a lot about all the extreme roots just looking at the extreme roots of the even-indexed Eulerian polynomials.

\begin{convencion}[Even case]
All in all, this convenient restriction on the parity of the indices we analyze means that here we will explicitly prove that the bounds get apart for $n=2m$ even but we only do this for practical reasons related to the shape of the sequence of vectors that we will manage, which will also ease our computations.
\end{convencion}

Additionally, in order to avoid repeating much of the work made in \cite{ale1}, we will not optimize a new $y$ and we will use directly the value obtained there. This means that we leave as an exercise for the interested reader to compute the optimal $y$ for our new sequence of vectors. The procedure to do this is not different from the one followed in \cite{ale1}. Nevertheless, for our objectives here in this article, the $y$ optimized in \cite{ale1} already fulfills what we want.

In particular, we will see that the sequence of differences between the sequence of bounds obtained for this $y$ and the univariate bound already goes to infinity with exponential growth when $n$ goes to infinity over the even-indexed terms. That is, even without the optimal $y$ for this case, we already obtain that the multivariate relaxation is explosively (exponentially) better than the univariate one when we want to measure that improvement as accuracy in the bounds for the extreme roots of the univariate polynomials lying in the diagonal (the univariate Eulerian polynomials).

\begin{observacion}[Previous optimal fulfills]
The optimal value obtained for the first entry $y$ (over which we had to maximize the absolute value of the multivariate bound) in \cite{ale1} will be used also here. This will save us a lot of work and will already be enough for our purposes of certifying an explosive (exponential) improvement in this article with respect to the bound obtained in \cite{ale1}.
\end{observacion}

After many numerical experiments, we are able to identify a candidate sequence of vectors for a good guess. In particular, we claim that the multivariate bound obtained through the sequence of vectors $$v:=\{(y,(-2^{m-i})_{i=3}^{m},(0,\frac{1}{2}),(1)_{i=1}^{m})\in\mathbb{R}^{n+1}\}_{n=1}^{\infty}$$ diverges apart from the corresponding optimal univariate bound. To see this, we have to compute terms again, as in \cite{ale1}.

\begin{remark}[Room for improvement]
It is clear that the improvement cannot occur at the first term of the asymptotic growth in the asymptotic scale used in \cite{ale1}. This cannot happen because the bounds obtained already in \cite{ale1} match the first asymptotic growth term of the extreme roots in this asymptotic scale. We therefore take differences because this decoupling of the bounds will occur at some other term deeper in the asymptotic scale.
\end{remark}

We focus now on our new sequence and name things accordingly. We make this clear in the following notation.

\begin{notacion}
We center our attention into the sequence of vectors $$v:=\{(y,(-2^{m-i})_{i=3}^{m},(0,\frac{1}{2}),(1)_{i=1}^{m})\in\mathbb{R}^{n+1}\}_{n=1}^{\infty}$$ and therefore from now on $\mult_{v}$ is the bound obtained through this new linearization of the relaxation.
\end{notacion}

We obtained this sequence of vectors through numerical experimentation. We briefly explain the main points and tasks involved in the process of realizing the experiments that allowed us to find this new sequence of vectors.

\begin{experimento}[Guessing the good sequence of vectors]
We compute the relaxation for small instances of the multivariate Eulerian polynomials, that is, setting $n=1,2,3,\dots$. The first thing we do is looking at the relaxation along the diagonal setting all the variables equal. Then we solve the corresponding generalized eigenvalue problem that this generates and take the rightmost (biggest) generalized eigenvalue, which is negative because we began with a relaxation of the original RCS and univariate Eulerian polynomials have negative roots. Then we compute the generalized eigenvector associated to such generalized eigenvalue. With this information, we can plot the entries of these eigenvectors equally spaced in the interval $[0,1]$. When we do this for several small values of $n$, we obtain the Figure \ref{figura}. \begin{figure}[ht]\label{figura}
\centering
\scalebox{.75}{\includegraphics{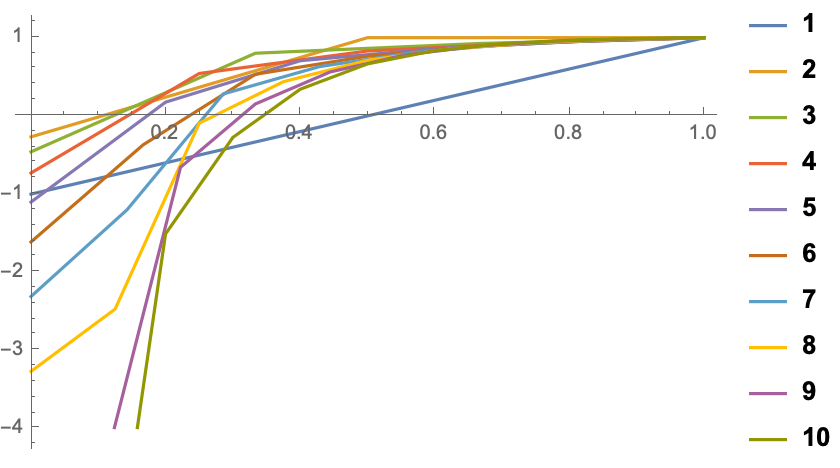}}
\caption{Representation of the entries of the eigenvectors in the interval $[0,1]$ for the relaxation corresponding to the multivariate Eulerian polynomials of degrees 1 to 10.}
\end{figure} The observation of this figure (and the generalized eigenvectors that form it) first shows us that we lack real control over what the first entry does. This is why we put there a $y$ that we will have to compute optimizing later. Now, for the rest of the entries we see that they grow from negative to positive. Passing from negative to positive, there is a jump that we identify as from $0$ to $\frac{1}{2},$ approximately. Finally, the last entries, these lying in the tails of the eigenvectors, seem to stabilize towards $1$. Thus, this analysis only lets open the question of how the growth in the head happens. Fortunately, it is easy to see that it goes through reducing the absolute value by powers of $2$. And this is how these experiments allowed us to construct a vector good enough to give us a precise enough approximation to the generalized eigenvector we need in order to linearize our problem and obtain the better asymptotic bounds that we will compute in this article. Ideally, we would have to compute the optimal value of $y$ depending on $n$. However, as we said, it turns out that the optimal $y$ computed before in \cite{ale1} is already good enough for our objectives here. This fact will simplify our work here by saving us from doing again these tedious optimization tasks in this article.
\end{experimento}

After seeing how we perform these experiments to search for nice enough choices for our sequence of linearizing guesses of approximated generalized eigenvectors, things are finally set in motion. The time of the computations can begin. First, we collect in the next section a few points to have in mind while proceeding with the aforementioned computations.

\section{Limitations, preparations, optimization and comparisons}\label{limit}

As we said, the improvement is limited in the asymptotic terms possible to improve since we know that the univariate bound already matches the first asymptotic growth term in the asymptotic scale introduced in Convention \ref{expsca} we are using here. We hope that this vector gives a difference with a growing term, contrary to what happened in \cite{ale1}, but this difference will therefore grow slower than $2^{n+1}$.

\begin{remark}[The growth of the difference]
As the absolute value of the leftmost root of the univariate Eulerian polynomials has dominant asymptotic term $2^{n+1}$ and the univariate relaxation also share this first term asymptotics growth, the fact that going multivariate can only improve the relaxation implies that the difference between these two bounds will have first term asymptotic growth vanishing against $2^{n}.$ This put a limit on how much better we can get precisely because we are already so close to the growth of the actual root just using the univariate relaxation.
\end{remark}

As in \cite{ale1}, we will run into software problems when we have to deal with expressions showing radicals and their differences. As there, we will have to use the power of conjugates in order to determine the actual surviving terms.

\begin{remark}[Management of radicals]
When doing asymptotics estimations using computer software it can happen that an expression of the form $c-\sqrt{b}$ verifies that the first asymptotic growth term $da^{n}$ of $c$ and $\sqrt{b}$ is the same. Therefore the way to deal with this correctly is by using conjugates so that we write $c-\sqrt{b}=\frac{(c-\sqrt{b})(c+\sqrt{b})}{c+\sqrt{b}}=\frac{c^{2}-b}{c+\sqrt{b}}$ so that we can now compute the asymptotic growth of the difference of squares $c^{2}-b$ without ignoring terms that remain alive while $c+\sqrt{b}\sim da^{n}+da^{n}=2da^{n}.$ This is the correct way of proceeding. This necessary procedure will make some computations really tedious and cumbersome, but there is no real way around this when we want to correctly estimate the growth of our sequences without ignoring terms that are alive deep into the asymptotic expansions of the sequences whose difference we are studying.
\end{remark}

If we do not proceed as above, the problems can be fatal. We can disregard asymptotic growth terms that are actually still strong and alive, making our computations wrong very easily.

\begin{warning}[Killing growth in differences of asymptotic terms]
In the case exposed in the remark above, to determine correctly the asymptotic growth of $c-\sqrt{b}$ we cannot just resort to the natural comparison by difference since it is not true in general that $c-\sqrt{b}\sim0$ in this case. What happens is that doing this naive (and wrong) difference of first term asymptotic growths we have ignored many asymptotic terms that go deeper into the expression of the asymptotic growth in the chosen scale.
\end{warning}

The procedure to check the first growth term of the huge expressions we encounter also deserves some noticing. It is a simple but tedious and dense task because we did not find Mathematica able to do that estimation of growth automatically in a symbolic way. For this reason, we had to do it manually checking slowly possible growth terms for very lengthy expressions. We expose here the best strategy we found to deal with this tedious task so that the reader can try (or improve) this process in case of interest or necessity.

%Puedes incluso mas bonito y mejor. Posibilidad de expansion.
\begin{procedimiento}[Checking asymptotic growth terms in lengthy Mathematica expressions]\label{proce}
As we noted above, the (correct) method of killing equal asymptotic terms using conjugates implies that expressions that are already lengthy will become even more so due to the necessity of taking squares of these expressions. In this case, the software is not good at finding the first asymptotic growth term and we had to do this tedious job manually. Moreover, in the next collection step, the software might collect several times against the same power of $2^{n}$, which introduces cancellations problems that one has to consider. All in all, with all this in mind, there are five simple but tedious steps to do in order to perform the computation of the first asymptotic growth term of these expressions in a reasonable amount of time. 

\textbf{1. Collection.} The expression \texttt{expr} initially obtained should be collected against the term $2^{n}$ using the command \texttt{Collect[expr,$2^n$]} so that such expression gets transformed into a sum of subexpressions $$\mbox{\texttt{expr}}=\sum_{i=1}^{s}\mbox{\texttt{expr}}_{i}$$ in which, generally, we would expect to have a different power of $2^{n}$ appearing in each subexpression $\mbox{\texttt{expr}}_{i}$, but in the actual computation this does not necessarily happen and a power of $2^{n}$ can appear in different subexpressions. The most important thing however is that, for each $i$, there is just one power of $2^{n}$ present in the expression $\mbox{\texttt{expr}}_{i}$. This easy fact simplifies and speeds up computations because the machine (and the reader) can therefore kill terms at first sight and using fewer instances. Surprisingly enough, the collection process using the command \texttt{Collect} is done fast by Mathematica. Moreover, from now on, instead of looking at the very long expression \texttt{expr}, we can look one by one at the elements of the tuple of subexpression $(\mbox{\texttt{expr}}_{1},\dots,\mbox{\texttt{expr}}_{s})$ thus obtained individually, which are \textit{just long} expressions.

\textbf{2. Exploration.} After performing this collection (a fast task), it will be easier to identify the subexpressions $\mbox{\texttt{expr}}_{i}$ having \textit{perceived} strongest growth terms and dominance. This does not mean that this exercise will be a fast task. However, under the collected terms, it is at least something that a human can perform in a reasonable amount of time simply exploring through $i\in\{1,\dots,s\}$. We can notice here that the usual pattern tends to put the subexpressions with dominant growth term under the indices $i$ lying around the middle of the ordered set $\{1,\dots,s\}$. So, in the collection of our initial expression by $2^{n}$ given by applying \texttt{Collect[expr,$2^n$]}, we are going to be looking for the dominant growth terms of the subexpressions $\mbox{\texttt{expr}}_{i}$ with $i\approx\frac{s}{2}$. But this is just an estimation, a good guess to begin exploring for the asymptotic growth term we search for. We might actually end up pretty far from that initially chosen expression $\mbox{\texttt{expr}}_{i}$ when we continue with our next points in this process of determination of the actual growth term of our whole expression \mbox{\texttt{expr}}. Once we have detected our suspected growth term, we can pass to next step.

\textbf{3. Estimation.} After we identify some simple and suitable candidate $i\approx\frac{s}{2}$ for an index containing the growth term of our whole expression \mbox{\texttt{expr}} in the collection \texttt{Collect[expr,$2^n$]} by $2^{n}$, we choose now inside the corresponding subexpression $\mbox{\texttt{expr}}_{i}$ the one term having the highest \textit{perceived} absolute growth in that subexpression. This is in fact easy. This term will therefore exhibit an asymptotic growth of the form $2^{kn}3^{pn}n^{r}$ for $k,p,r\in\mathbb{Z}$ giving an expression of our previously established and convened asymptotic scale in Convention \ref{expsca}. In order tom check that we are right with our choice and perception of this as a good candidate to establish the asymptotic growth of the whole expression, we can now first compute the limit of the chosen subexpression $\mbox{\texttt{expr}}_{i}$ of the collection \texttt{Collect[expr,$2^n$]} by $2^{n}$ against this term $2^{kn}3^{pn}n^{r}$ to ensure that we took actually the correct highest growth term in absolute value within such subexpression. This will give us $$\lim_{n\to\infty}\frac{\mbox{\texttt{expr}}_{i}}{2^{kn}3^{pn}n^{r}}=l.$$ If the limit $l$ is finite and not $0$, we are right in our choice. Otherwise, the choice we did is wrong and we have to revise it until we find something that provides us with a finite limit $l\in\mathbb{R}\smallsetminus\{0\}$. In general, if $l=0$, then the growth of the absolute value of the expression we look for should be smaller; if, otherwise, $l=\pm\infty$, then the growth of the absolute value of the expression we look for should be bigger. We proceed checking our choices in this way until we find a term that provides us with a finite limit $l\in\mathbb{R}\smallsetminus\{0\}$. With this choice of term of our asymptotic scale $2^{kn}3^{pn}n^{r}$ we can now pass to the next point.

\textbf{4. Taking limit.} Here we compute the limit growth of the whole expression \texttt{expr} against the term determined above $2^{kn}3^{pn}n^{r}$ in the asymptotic scale. We compute over the subexpressions so that we compute first the tuple of limits $$\left(\lim_{n\to\infty}\frac{\mbox{\texttt{expr}}_{i}}{2^{kn}3^{pn}n^{r}}\right)_{i=1}^{s}.$$ Then we keep subexpressions $\mbox{\texttt{expr}}_{i}$ corresponding to indices $i$ with $\frac{2^{in}\mbox{\texttt{expr}}_{i}}{2^{kn}3^{pn}n^{r}}=\pm\infty$ in case there are some, delete the rest and go to the previous case because there is therefore a higher growth competitor in these surviving expressions. This is so because, if the limit gives $\pm\infty$, then the term $2^{kn}3^{pn}n^{r}$ has small absolute growth for other subexpressions and therefore it cannnot be the correct growth term unless there are cancellations. These cancellation could well happen and we will learn how to deal with them soon. In this case of cancellations, we can delete copies of the estimated term inside subexpressions of the collection we chose in the third step and proceed to do another estimation on the remaining of the expression. This will be a fundamental step of the remaining case. This remaining case happens when all these limits are finite and sum up to a finite number $w\in\mathbb{R}$. There are two options. If that $w\neq0$, we are done and we can pass to the final step because we actually know then that the first growth term of our whole expression \texttt{expr} is therefore $w2^{kn}3^{pn}n^{r}$. Otherwise, the sum of all the limits above gives $w=0$ and the term $2^{kn}3^{pn}n^{r}$ suffers an actual cancellation along the whole original expression \texttt{expr}. This means that its growth is actually too big in absolute value growth for the expression and that it is actually not a term of the whole expression when expanded since it gets canceled out in such expansion. This cancellation means therefore that this term $2^{kn}3^{pn}n^{r}$ does not give the correct growth term. In this case, we go back to the original collection $$\mbox{\texttt{expr}}=\sum_{i=1}^{s}\mbox{\texttt{expr}}_{i}$$ and we modify the subexpressions to ensure that these terms are eliminated so they do not disturb our computations again. In particular, in these subexpression $\mbox{\texttt{expr}}_{i}$ for which we have $$\lim_{n\to\infty}\frac{\mbox{\texttt{expr}}_{i}}{2^{kn}3^{pn}n^{r}}\neq0$$ we delete all terms of the form $2^{kn}3^{pn}n^{r}$ so that we form in this way the new subexpression $\mbox{\texttt{expr}}_{i}'$. In this way, we kill all the terms in the subexpressions $\mbox{\texttt{expr}}_{i}$ that have anything to do with this canceled out term $2^{kn}3^{pn}n^{r}$ that is actually not appearing in the expansion of the whole expression \texttt{expr}. After finishing this process, we proceed to do another estimation in the third step on the remaining of the expression. This process forces us to use a different collected term in this new application of this step. We go therefore back to the third point. We continue doing this until the loop closes, which, in our examples, happened very fast, although the computations of the limits took some time computational time.

\textbf{5. Cut and kill.} After several iterations of this method (which were not more than two in our examples), we find a term $2^{kn}3^{pn}n^{r}$ that, when compared in the limit with the original expression \texttt{expr} we had, gives a finite nonzero value $w$. Our initial expression \texttt{expr} has therefore first growth term given by $w2^{kn}3^{pn}n^{r}$ and we have finished.

As we said, this process is tedious and lengthy, but this was the procedure that we found that allowed us to save the most computation, sanity and time. We did not find how to correctly do this automatically in Mathematica with such lengthy expressions as the ones we have to deal with. Although this is almost equivalent to searching for the fastest growing term in each collection against $2^{n}$ and then determine the biggest non-annihilated one by cancellation. However, this last procedure would imply more workload for the human because it would require to determine the growth for all the terms in the collection instead of trying to guess one each time until the limit does not kill everything (all terms going to $0$) or gets surpassed by something (some term going to $\infty$ in absolute value). An abstract explanation for the process would be to separate the collections by $2^{n}$ in a tuple, then choose some term $2^{kn}3^{pn}n^{r}$ that we could suspect to be the growth and compute the limit with respect to all the entries in the tuple only keeping those entries that do not get annihilated. If all are finite, then we just add them and if that operation is nonzero, that is our growth; otherwise we eliminate all these terms and proceed again. If some surviving term is infinite, we try again just looking at the infinite collection terms and its strongest terms without forgetting that some of these could actually cancel out so we have to keep in memory the full tuple in case we have to eliminate terms in the original tuple to go back to the previous step. As we see, this is indeed a simple but tedious and lengthy process that works but requires high doses of patience and stamina to go forth until the end of the exercise without losing sanity, calm, hope and serenity in the middle of any of the tasks involved in its complete resolution.
\end{procedimiento}

Fortunately, the readers do not have to go through the determination of these growths terms using the process described above for our examples here because we did this task for them. Once the growth term is established and you have it written down somewhere (as we collected them in this article), checking that these are the correct growth terms of the whole expressions involved is an easy task. In particular, this task just requires computing a limit. This can take some time but it is easily doable. Taking our next advice is doable even much faster.

\begin{advice}
This checking process is faster if the limit is taken over the expression written as a collection against $2^{n}$, as we did above.
\end{advice}

Otherwise it takes more time for Mathematica to do that limit because the program does not do the collection automatically and seems to get loss in a never-ending computation that actually finish much faster in the way we mention in our advice above. As before then, if we fail to proceed like we mentioned here, many problems can arise. In particular, we can run into highly inefficient uses of our time, energy and computing and with programs that run for too long unnecessarily. Using the simple tricks exposed above will thus save not just our sanity.

\begin{warning}[Extremely lengthy expression and time-consuming computations]
If we do not direct the program (and ourselves) to act through the collection step \texttt{Collect[expr,$2^n$]}, Mathematica do not do it automatically by itself and the program gets lost in never-ending computations when it tries to figure out the limit against the selected term in the asymptotic scale that we could suspect encodes the actual asymptotic growth of the expression \texttt{expr}. For this reason, it is extremely important not to jump over this step. Moreover, the elimination steps of Procedure \ref{proce} performed when some terms cancel out or fail to survive the limit also speed up the whole process with each iteration. This speedup is something that our sanity, our computer and ourselves will desire \textit{with our whole soul} while performing this task.
\end{warning}

Under normal circumstances, we would also have to optimize over the parameter $y$ finding a nice enough sequence providing an optimal bound to the extreme roots of the univariate Eulerian polynomials. However, we can skip this optimization process here because we do not need to do it for our objectives claimed in this article. But it is also doable through exactly the same process followed and showed in \cite{ale1}. Moreover, in the comparison process we will have to take care of expressions with radicals that will force us again to proceed with chained conjugations. That is, conjugations within conjugations.

\begin{remark}[Chain conjugations]
In some instances, the roots could not completely disappear after a first conjugation. In these cases, several rounds of conjugation might be necessary and each of these determines a growth that has to be treated with enough caution. This will happen, e.g., when there two different radical expressions $\sqrt{q}$ and $\sqrt{g}$ with $g$ and $q$ not easy to relate by algebraic means.
\end{remark}

Beware of the problems that can appear as a consequence of not following this process correctly. These problems are not easy to skip in such a tedious task.

\begin{warning}[Conjugate with care]
A failure while dealing with these (chained) conjugations could imply that we ignore asymptotic growth terms at some point in the chain process that actually survive. This may significantly and wrongly alter the whole computation producing a spurious and unreliable result about the asymptotic growth we are trying to figure out.
\end{warning}

These cautions will be helpful in order to have a smooth path along the next final sections. Following these simple advices will land us swiftly and safely into the ground our desired results.

\section{Form of the bound obtained through $(y,(-2^{m-i})_{i=3}^{m},(0,\frac{1}{2}),(1)_{i=1}^{m})\in\mathbb{R}^{n+1}$, pseudoptimization of $y$ and management of radicals}\label{form}

As in \cite{ale1}, we first have to compute the form of the linearization in order to analyze the bound. This computation is the content of this section. We recall, once again, that $n=2m$ so we are just looking at even indices (and degree) in the sequences of Eulerian polynomials.

\begin{proposicion}[Form of bound obtained for arbitrary $y$]\label{newboundcomplicatedvector}
The bound for the extreme root of the univariate Eulerian polynomial obtained through linearization of the spectrahedral relaxation applied to the corresponding multivariate Eulerian polynomial $A_{n}$ by the vector $(y,(-2^{m-i})_{i=3}^{m},(0,\frac{1}{2}),(1)_{i=1}^{m})\in\mathbb{R}^{n+1}$ is of the form $x_{n,r}\geq-\frac{D}{N}$ with $D=$ \begin{gather*}
    -\frac{1}{12} + 2^{3m} + 2^{2 + m} + 5 \cdot 2^{-3 + 2m} - 7 \cdot 2^{-1 + 2m} + 3 \cdot 2^{1 + 3m} - 2^{3 + 3m}\\ + \frac{1}{3} \cdot 2^{2 + 4m} + \frac{1}{3} \cdot 2^{3 + 4m}
 - 2 \cdot 3^{-1 + m} - 2^{4 + m} \cdot 3^{-1 + m} + 3^m - 2^{1 + m} \cdot 3^m\\ - 3^{1 + m} + 2^{2 + m} \cdot 3^{1 + m}
 - 2 \cdot 3^{1 + 2m} - \frac{11 \cdot 4^{-2 + m}}{3} + m - 2^{3m} \cdot m\\ - 5 \cdot 2^{-4 + 2m} \cdot m - 2^{-3 + 2m} \cdot m
 + 2^{-1 + 2m} \cdot m + 4^{-2 + m} \cdot m + 2^{-4 + 2m} \cdot m^2\\ (- 3 + 2^{-1 + m} + 2^{1 + m}
 - 2^{2 + m} + 2^{2 + 2m} - 2m - 2^{-1 + m} \cdot m)y + 2my^2
\end{gather*} and $N=$ \begin{gather*}
    \frac{1}{12} + 2^{2m} - 2^{3m} + 5 \cdot 2^{4m} - 2^{2 + m} + \frac{11}{3} \cdot 2^{-4 + 2m} - 5 \cdot 2^{-3 + 2m} - 7 \cdot 2^{-1 + 2m} \\
 + 3 \cdot 2^{1 + 2m} + 9 \cdot 2^{-3 + 3m} - 47 \cdot 2^{-2 + 3m} + 3 \cdot 2^{-1 + 3m} - 2^{2 + 3m} - \frac{1}{7} \cdot 2^{3 + 3m} \\
 + \frac{1}{7} \cdot 2^{4 + 3m} + \frac{5}{7} \cdot 2^{5 + 3m} - \frac{27}{5} \cdot 2^{-3 + 4m} + 5 \cdot 2^{-1 + 4m} - 2^{1 + 4m} + 2^{1 + 5m} \\
 + 3 \cdot 2^{2 + 5m} - 2^{4 + 5m} + \frac{1}{7} \cdot 2^{3 + 6m} + \frac{1}{3} \cdot 2^{4 + 6m} + \frac{1}{21} \cdot 2^{5 + 6m} + 2 \cdot 3^{-1 + m} \\
 - 11 \cdot 2^{2m} \cdot 3^{-1 + m} - \frac{1}{5} \cdot 2^{3 + m} \cdot 3^{-1 + m} + 2^{4 + m} \cdot 3^{-1 + m} + 13 \cdot 2^{2 + 2m} \cdot 3^{-1 + m} \\
 - 2^{5 + 3m} \cdot 3^{-1 + m} - 3^m - 2^{-1 + m} \cdot 3^m + 2^{1 + m} \cdot 3^m + 7 \cdot 2^{-1 + 2m} \cdot 3^m \\
 - 2^{2 + 3m} \cdot 3^m + 3^{1 + m} - 2^{2 + m} \cdot 3^{1 + m} - 2^{3 + 2m} \cdot 3^{1 + m} + 2^{3 + 3m} \cdot 3^{1 + m} \\
 - 2^{1 + 2m} \cdot 3^{2 + m} + 2^{-2 + 2m} \cdot 3^{3 + m} + 4 \cdot 3^{1 + 2m} - 2^m \cdot 3^{1 + 2m} - 2^{1 + m} \cdot 3^{1 + 2m} \\
 + 2^{2 + m} \cdot 3^{1 + 2m} - \frac{1}{5} \cdot 2^{2 + 2m} \cdot 3^{1 + 2m} + 6^m - 6^{1 + m} - \frac{13}{5} \cdot 6^{1 + 2m} \\
 - m - 2^{2m} \cdot m + 2^{3m} \cdot m + 5 \cdot 2^{-3 + 2m} \cdot m + 2^{-2 + 2m} \cdot m - 5 \cdot 2^{-2 + 4m} \cdot m \\
 - 2^{-1 + 4m} \cdot m + 2^{1 + 4m} \cdot m - 2^{1 + 5m} \cdot m + 2^{1 + 2m} \cdot 3^{-1 + m} \cdot m + 2^{-2 + 2m} \cdot 3^m \cdot m \\
 - 2^{-1 + 2m} \cdot 3^{1 + m} \cdot m + 2^{-1 + m} \cdot 3^{1 + 2m} \cdot m - 2^{-4 + 2m} \cdot m^2 + 2^{-3 + 4m} \cdot m^2 \\
 + y(3 - 3 \cdot 2^{-1 + m} + 2^m + 5 \cdot 2^{3m} + 2^{1 + m} - 2^{2 + 2m} + 2^{1 + 3m} \\
 - 2^{3 + 3m} + 2^{3 + 4m} - 2^{4 + m} \cdot 3^{-1 + m} - 2^{1 + m} \cdot 3^m + 2^{2 + m} \cdot 3^{1 + m} \\
 - 4 \cdot 3^{1 + 2m} + 2m + 2^{-1 + m} \cdot m - 2^{3m} \cdot m) +y^{2}(- 2 + 2^{1 + 2m} - 2m).
\end{gather*}
\end{proposicion}

\begin{proof}
The bound obtained through this new vector comes from an equation of the form $D+xN\geq 0$ with $D=$ \begin{gather*}y(yL_{p}(1)-\sum_{i=3}^{m}2^{m-i}L_{p}(x_{i-2})+\frac{1}{2}L_{p}(x_{m})+\sum_{i=m+1}^{n}L_{p}(x_{i}))-\\\sum_{j=3}^{m}2^{m-j}(yL_{p}(x_{j-2})-\sum_{i=3}^{m}2^{m-i}L_{p}(x_{i-2}x_{j-2})+\\\frac{1}{2}L_{p}(x_{m}x_{j-2})+\sum_{i=m+1}^{n}L_{p}(x_{i}x_{j-1}))+\\\frac{1}{2}(yL_{p}(x_{m})-\sum_{i=3}^{m}2^{m-i}L_{p}(x_{i-2}x_{m})+\\\frac{1}{2}L_{p}(x_{m}x_{m})+\sum_{i=m+1}^{n}L_{p}(x_{i}x_{m}))+\\\sum_{j=m+1}^{n}(yL_{p}(x_{j})-\sum_{i=3}^{m}2^{m-i}L_{p}(x_{i-2}x_{j})+\frac{1}{2}L_{p}(x_{m}x_{j})+\sum_{i=m+1}^{n}L_{p}(x_{i}x_{j}))
\end{gather*} and $N=$ \begin{gather*}y(y\sum_{k=1}^{n}L_{p}(x_{k})-\sum_{i=3}^{m}2^{m-i}\sum_{k=1}^{n}L_{p}(x_{i-2}x_{k})+\\\sum_{k=1}^{n}\frac{1}{2}L_{p}(x_{k}x_{m})+\sum_{i=m+1}^{n}\sum_{k=1}^{n}L_{p}(x_{k}x_{i}))-\\\sum_{j=3}^{m}2^{m-j}(y\sum_{k=1}^{n}L_{p}(x_{k}x_{j-2})-\sum_{i=3}^{m}2^{m-i}\sum_{k=1}^{n}L_{p}(x_{i-2}x_{j-2}x_{k})+\\\sum_{k=1}^{n}\frac{1}{2}L_{p}(x_{k}x_{j-2}x_{m})+\sum_{i=m+1}^{n}\sum_{k=1}^{n}L_{p}(x_{j-2}x_{k}x_{i}))+\\\frac{1}{2}(y\sum_{k=1}^{n}L_{p}(x_{k}x_{m})-\sum_{i=3}^{m}2^{m-i}\sum_{k=1}^{n}L_{p}(x_{i-2}x_{m}x_{k})+\\\sum_{k=1}^{n}\frac{1}{2}L_{p}(x_{k}x_{m}x_{m})+\sum_{i=m+1}^{n}\sum_{k=1}^{n}L_{p}(x_{m}x_{k}x_{i}))+\\\sum_{j=m+1}^{n}(y\sum_{k=1}^{n}L_{p}(x_{k}x_{j})-\sum_{i=3}^{m}2^{m-i}\sum_{k=1}^{n}L_{p}(x_{i-2}x_{j}x_{k})+\\\sum_{k=1}^{n}\frac{1}{2}L_{p}(x_{k}x_{j}x_{m})+\sum_{i=m+1}^{n}\sum_{k=1}^{n}L_{p}(x_{j}x_{k}x_{i})).\end{gather*} Computing these last expressions finishes the proof.
\end{proof}

As in \cite{ale1}, now we would only need to choose $y$ carefully enough so that this expression gives us an (explosive exponential) improvement with respect to previously computed bounds. It turns out that the $y$ computed in \cite{ale1} already does this job for us here.

\section{Managing radicals again and a comment on optimality}\label{manag}

As we told above, in this section we will consider the optimal $y$ computed for the case considered in \cite{ale1}. Now, as happened in \cite{ale1}, we will have to deal with the problems it introduces through the radicals in its expression. We basically take the $y$ in \cite{ale1} but, for more clarity in our expressions of the sequence of vectors, it is better if $y$ actually equals the opposite of the $y$ computed in \cite{ale1}. Otherwise we would have to change many signs in the expression of our sequence of vectors. We check that this choice is appropriate for us.

\begin{proposicion}[Adequacy of our choice of $y$]
Taking $y=\frac{b+\sqrt{b^2-4ac}}{2a}$ with $a,b,c$ as in Lemma \ref{formofboundfirstwinner} but substituting $n=2m$ makes $N>0$ and $D>0$.
\end{proposicion}

\begin{proof}
We remind the reader to observe the change of sign in the numerator of our new $y$. We study now the conditions on $y$ so $N>0$ for $n$ big enough, as was done in \cite{ale1}. Now we already know from \cite{ale1} that we can express the first asymptotic growth term of our $y$ as $y\sim-\frac{3^{n+1}}{2^{n+1}n}$. Now the winning terms (in each power of $y$ in the expression obtained in Proposition \ref{newboundcomplicatedvector}) are $2^{6 m + 3}=8^{n+1}, 2^{4 m + 3}y=4^{2m+1}2y=4^{n+1}2y$ and $2^{1+2m}y^2=2^{n+1}y^2$. Thus the only term that would represent a problem is $4^{n+1}2y$ because it is negative but $4^{n+1}2y\sim -4^{n+1}2\frac{3^{n+1}}{2^{n+1}n}=-6^{n+1}2$, which fades against the mighty power of the exponential $8^{n+1}$. Thus, our $y$ is adequate. For checking that $D\neq0$ for this $y$, we have to do exactly the same as in \cite{ale1} and the result is immediate. This settles the adequacy of this choice of the sequence $y$.
\end{proof}

Once we have established that our choice of the sequence $y$ is adequate, we are sure that we have a bound. We want to see that this bound is not just better, but \textit{much} better than the previous bounds. This is the content of the next result, which is the analogue of \cite[Lemma 22.1]{ale1} for this new choice of a linearizing sequence of vectors. In absolute value, our improved root bound is now $|q_{1}^{(n)}|\geq\frac{N}{D}:=\mult_{v(n)}(n)$ and we require again $\un(n)<\mult_{v(n)}(n)$ with the twist that now we want this difference to grow with $n=2m$.

\begin{lema}[Comparison with univariate bound]
For the subsequences of even $n=2m$ indices, we have the asymptotic growth of the difference of bounds $$\mult_{v(n)}(n)-\un(n)\sim\frac{3}{8}\left(\frac{9}{8}\right)^m.$$
\end{lema}

\begin{proof}
We proceed now as in \cite{ale1} taking special care when freeing radicals and using the strategy of considering conjugates. In particular, we compute $\mult_{v}-\un$. We proceed again step by step writing $y=\frac{f+\sqrt{g}}{h}$ so $y^2=\frac{f^2+g+2f\sqrt{g}}{h^2}$ and substitute these values in the expression of $\mult_{v}=\frac{N}{D}$ and kill denominators inside $N$ and $D$ writing $\mult_{v}=\frac{N(y)}{D(y)}=\frac{h^{2}N(y)}{h^{2}D(y)}$ in order to obtain the manageable expression $\mult_{v}=\frac{\alpha+\beta\sqrt{g}}{\gamma+\delta\sqrt{g}}.$ Remembering again that $\un=\frac{p+r\sqrt{q}}{s}$, we develop conveniently the expression $\mult_{v}-\un=$ $$\frac{k+v+u+w}{s(\gamma+\delta\sqrt{g})}$$ with $k:=s\alpha-p\gamma,v:=(s\beta-p\delta)\sqrt{g},u:=-r\gamma\sqrt{q}$ and $w:=-r\delta\sqrt{gq}$. The dominant terms of the sums in our decomposition of the numerator are \begin{enumerate}
\item $k\sim -2^{22 m+19} 3^{2 m+1} m^4$,
\item $v\sim(3^{2 m+1} 4^{8 m+7} m^3)\sqrt{2^{6 n+8} n^{2}}=2^{22 m+19} 3^{2 m+1} m^4$,
\item $u\sim-(-2^{16 m+15} 3^{2 m+1} m^3)\sqrt{2^{6n+6} n^{2}}=-2^{22 m+19} 3^{2 m+1} m^4$,
\item $w\sim(3^{2 m+1} 4^{5 (m+1)} m^2)\sqrt{(2^{6 n+8} n^{2})(2^{6n+6} n^{2})}=2^{22 m+19} 3^{2 m+1} m^4$.\end{enumerate} As happened in \cite{ale1}, these dominant terms annihilate so we have to keep track again of the surviving terms freeing radicals through the use of conjugates. We want to be able to estimate correctly the dominant behaviour of both the numerator and the denominator and, as in \cite{ale1}, we have the luck that the conjugate expressions in the numerator $(k+u)-(v+w)\sim$ \begin{gather*}-2^{22 m+20} 3^{2 m+1} m^4-2^{22 m+20} 3^{2 m+1} m^4=-2^{22 m+21} 3^{2 m+1} m^4\end{gather*} and in the denominator \begin{gather*}\gamma-\delta\sqrt{g}\sim 2^{12 m+12} m^3-(-2^{6 m+7} m^2)\sqrt{2^{12 m+10} m^2}=\\2^{12 m+12} m^3-(-2^{6 m+7} m^2)(2^{6 m+5} m)=2^{12 m+13} m^3\end{gather*} do not annihilate. We use the good behaviour of these conjugates to free some radicals via multiplication. We proceed similarly with the denominator. This is what we do next.

As managing the denominator is easy, we proceed first with the numerator. We do this multiplying by the nice conjugates we saw above. This will free some radicals. We warn the reader that in the next explanation of the structure of the computations involved all the names are locally set. As we have seen that $(k+u)-(v+w)$ is dominantly $-2^{22 m+21} 3^{2 m+1} m^4$, we can multiply by it in order to get a nice difference of squares. Hence, multiplying the numerator $k+u+v+w$ by $(k+u)-(v+w)$ we obtain precisely $(k+u)^2-(v+w)^2=k^2+u^2+2ku-v^2-w^2-2vw=(k^2+u^2-v^2-w^2)+(2ku)+(-2vw)=r+s+t.$ We will again need a further use of conjugates to determine the real growth of this because the expressions of $s$ and $t$ contain $\sqrt{q}.$ We have to see first what is the growth of the conjugate $r-(s+t)$. In order to see this, we first need to establish the real growth of $s+t$, which is not immediate because there is a cancellation again as it is easy to see that $ku\sim (2^{22 m+19} 3^{2 m+1} m^4)^2\sim vw$. Thus the use of conjugates is necessary again. It is clear that $s-t=(2ku)-(-2vw)=2(ku+vw)\sim 4(2^{22 m+19} 3^{2 m+1} m^4)^2=2^{44 m+40} 3^{4 m+2} m^8$ while computing we can see that $(s+t)(s-t)=s^2-t^2=(2ku)^2-(2vw)^2=4qr^{2}(((s\alpha-p\gamma)\gamma)^{2}-((s\beta-p\delta)g\delta)^{2})\sim-2^{83 m+76} 3^{10 m+5} m^{15}$ so we obtain that $$s+t=\frac{s^{2}-t^{2}}{s-t}\sim\frac{-2^{83 m+76} 3^{10 m+5} m^{15}}{2^{44 m+40} 3^{4 m+2} m^8}=-2^{39 m+36} 3^{6 m+3} m^7.$$ As $r\sim-2^{39 m+36} 3^{6 m+3} m^7$, we get that the sum $r+s+t\sim-2^{39 m+37} 3^{6 m+3} m^7$ and therefore we obtain $k+v+u+w=\frac{(k+u)^{2}-(v+w)^{2}}{(k+u)-(v+w)}\sim\frac{-2^{39 m+37} 3^{6 m+3} m^7}{-2^{22 m+21} 3^{2 m+1} m^4}=2^{17 m+16} 3^{4 m+2} m^3.$

Proceeding similarly with the denominator we get that the conjugate of the problematic factor is dominantly $\gamma-\delta\sqrt{g}\sim2^{12 m+13} m^3$ while the resulting difference of squares is dominantly $\gamma^2-\delta^{2}g\sim2^{24 m+25} m^5$ so, finally, we can see that the problematic factor is dominantly $\gamma+\delta\sqrt{g}\sim\frac{2^{24 m+25} m^5}{2^{12 m+13} m^3}=2^{12 m+12} m^2$ so, as $s\sim2^{8 m+7} 3^{2 m+1} m,$ the denominator is dominantly $$s(\gamma+\delta\sqrt{g})\sim2^{20 m+19} 3^{2 m+1} m^3.$$ All in all, we see that the difference is dominantly $$\frac{2^{17 m+16} 3^{4 m+2} m^3}{2^{20 m+19} 3^{2 m+1} m^3}=2^{-3 m-3} 3^{2 m+1}=\frac{3}{2^3}\frac{3^{2m}}{2^{3m}}=\frac{3}{8}\left(\frac{9}{8}\right)^m,$$ as we wanted to see.
\end{proof}

The consequences of this lemma are clear. In particular, this result shows us that we obtained the best bound so far using the relaxation (and in the literature). The analysis of this improvement in qualitative comparison to the results of \cite{ale1} is the content of the next section.

\section{Explosive improvement of the bounds}\label{explo}

We immediately obtain the next corollary. We will see that we could even go beyond now in a fairly natural way.

\begin{corolario}[Exponential explosion in the difference]
The difference between $\mult_{v(n)}(n)$ and $\un(n)$ grows exponentially when $n=2m\to\infty$ by steps of size $2$, that is, through the indices given by even numbers.
\end{corolario}

\begin{proof}
It is clear because $\frac{3}{8}\left(\frac{9}{8}\right)^m\to\infty \mbox{\ when\ } m\to\infty$ and this happens exponentially because $\frac{9}{8}>1.$
\end{proof}

As we see, computing the $y$ in \cite{ale1} gave us a hint for a $y$ to use in this more elaborated case. This choice of $y$ eventually allowed us to carry out the numerical experiment that gave us this last nice bound. It would have been much more difficult to find that good guess of a sequence of approximated generalized eigenvectors without first having a nice $y$ to try with in the computations. However, it is clear that we could have made our choice of $y$ here optimal with a bit more of work. For this, we could have directly proceeded similarly as it was done previously in \cite{ale1}.

\begin{remark}[Better $y$]
A straightforward exercise that should give now an even better bound is to compute the optimal $y$ for the current expression at hand as it was done in \cite{ale1}. We leave it as an exercise because we obtained already what we wanted here using the current suboptimal $y$: we showed that the bound obtained through the multivariate spectrahedral relaxation drastically (explosively) improves by an \textit{exponential} growth term with respect to the univariate bound obtained in \cite{ale1}.
\end{remark}

Thus, we have the next nice and easy exercise. Its solution would consist in repeating the whole optimization process performed in \cite{ale1} but for the new sequence of vectors $v$ we used to linearize the relaxation in this current paper.

\begin{ejercicio}[Finding such $y$]
Find the optimal $y$ for the new sequence of vectors $v$ used here, construct the corresponding bound obtained through the use of this new optimal $y$ and compare it with the bound obtained in this article using the suboptimal $y$ obtained in \cite{ale1}.
\end{ejercicio}

This finishes our analysis of the bounds obtained through the linearization of the spectrahedral relaxation using this new more structured guess of sequence of vectors. In particular, this allowed us to prove and certify that effectively the multivariate spectrahedral relaxation provides an improvement in the global approximation of the Eulerian RCSs measured around the diagonal that does not vanish at infinity. In fact, we can actually certify that this improvement grows rapidly (exponentially) towards infinity when the index $n$ goes to infinity.

Now we want to discuss how we could explore which further steps we could take to improve the bounds obtained through the spectrahedral relaxation powering it with other compatible methods traditionally used to compute bounds for the extreme roots of the univariate Eulerian polynomials. In particular, we will try to establish what \textit{compatibility with the relaxation} means and which other approximation methods could fulfill this compatibility requirement. For a comprehensive compendium on approximation methods for roots of univariate polynomial that we could explore, study and discuss in the future, we direct the reader to the two great volumes \cite{macnameeii}. Our explorations in the next section dive into the topic of the possibility of fully extending these univariate methods to multivariate ones. We want to explore doing this either directly or by using indeed the spectrahedral relaxation to provide \textit{a kind of extension of these univariate methods by an auxiliary technique} through the powerful multivariate management tool of the spectrahedral relaxation.

\section{Other venues for improvements}\label{other}

There are immediate venues for improvement of the results presented and exposed up until here. The most natural one is searching for even better guesses for sequences of vectors. We cannot discard the possibility of finding even better sequences of vectors. However, these sequences of vectors should also combine easy computability and understandable form with reliability and strength of the results and asymptotic bounds given. And such balance is a difficult one.

\begin{remark}[Better guesses]
A better guess of the structure of the approximated generalized eigenvectors should therefore combine treatability in the computations and actual improvement of the asymptotic estimations. We cannot discard that such structures exist, but, seeing the numerics of the actual eigenvectors computed for small instances, it seems difficult to know how the structure used here could be improved in these two directions. If we have to make a bet, the long tail of all ones in the second half of the structure of the vectors seems to, again as happened with the vector used in \cite{ale1}, agglutinate (and therefore kill) much of the information about the original polynomial variable split that is actually still preserved by the spectrahedral relaxation. The most obvious improvement of the structure of the guessed sequence of  approximated generalized eigenvectors towards a more effective retrieval of information about the multivariate split during the linearization of the spectrahedral relaxation should, in our view, be somewhere there or along these lines of thought and research.
\end{remark}

Additionally, the results here show us that multivariability is a dealbreaker at the time of obtaining bounds. Therefore it is natural for us to wonder if adding even more variables could improve our results further.

\begin{remark}[More variables]
We proved here that adding more variables in a \textit{meaningful} sense improves the spectrahedral relaxation. The meaningful sense in this case was a refinement in the counting of descents that kept track of the position of the descent tops within each permutation. However, we also saw at the beginning of this article that this multivariate lifting can also be done entirely synthetically through stability preserving operators and recurrence. Thus there are different mechanisms to increase the number of variables and we know that increasing the variables can improve the spectrahedral relaxation. The question therefore lies in measuring how much different methods do this variable expansion (increase) so that we can determine which of these many approaches are better for our objectives of increasing accuracy on the approximation given by the spectrahedral relaxation. All this while keeping RZ-ness (or other multivariate generalizations of being real rooted like stability or hyperbolicity). Notice the immediate relation of this line of research with the amalgamation conjecture proposed in \cite{main}, which is yet another proposed path to construct (more-variable) multivariate liftings. We point towards researching the theory of these multivariate lifting methods in general from the perspective of their approximation theory potential (also with respect to the spectrahedral relaxation accuracy power of each method).
\end{remark}

The problem is that it is not entirely clear how to add more variables keeping RZ-ness beyond the explained process of recursion through stability preserving operators. Another possibility lies precisely in changing the stability preserving operators used in the recursion, although this approach would imply the study of a different sequence of multivariate liftings. Thus, this method, by itself, raises many interesting questions in the direction of (the necessity of developing tools for) comparing the approximations, estimations and bounds obtained by different recursions and liftings of the original sequence of polynomials depending on the different stability preserving operators used in the recursions producing the corresponding liftings.

\begin{remark}[Other recursions]
Therefore the Eulerian rigidly convex sets studied here are not the only ones. Other recursions than these explained in Section \ref{intro} could be used instead. The product of these recursions would therefore be other families of RZ-polynomials having the (information about the) corresponding univariate Eulerian polynomials injected in their diagonals (or somewhere else in maybe some other sense). The only reason we focused our attention first in the multivariate lifting we used here is that such lifting is the most natural. This naturality comes from the fact that this lifting explored here has a combinatorial interpretation as a tagging of descent tops with variable indices, which is certainly a natural generalization of just counting descents towards a framework of counting features in a finer manner keeping and storing further information. This does not mean that other recursion through different stability preservers could not have other nice combinatorial interpretations or just better analytical behaviour when combined with the methods we used here. This therefore opens the door to exploring and studying these other possible different liftings generating different sequences of multivariate RZ polynomials. This also implies the construction and exploration of other sequences of RCSs that could be studied in connection to the original univariate Eulerian polynomials introduced centuries ago in \cite{euler}. Thus these new objects could shed more light into the structures still hidden within the realms of the generating classical ideas. In this way, these new objects, connected with the classical ideas they originate from, could help us in the task of showing new ramifications into the current theories where they appear combined nowadays. At the end, therefore, this strategy would produce new deep insights by approaching the subject through many varied mathematical branches in an integrated and systemic way around these polynomials and their properties. These diverse approaches have different theoretical tastes and they make use of recently studied, introduced and popularized methods. In particular, these tools can be put in connection with the Eulerian polynomials and topics adjacent to them across very diverse realms of current mathematical research.
\end{remark}

This automatically begs the question of which stability preserving operators for the recursivity work better with the spectrahedral relaxation in the described sense of producing better approximations to the corresponding RCSs. That concept of \textit{better approximation} was measured through bounds for the extreme roots of the univariate Eulerian polynomials injected in the diagonal in this article, but other more accurate and global measures could eventually do a better job in this sense. All in all, we want to study which stability preserving operators produce sequences of multivariate liftings for which the spectrahedral relaxation becomes an even more accurate approximation to the corresponding RCSs (not necessarily just around the diagonal).

\begin{remark}[Recursions and the diagonal]
In addition to that, we have to notice that our polynomials of interest do not necessarily have to land always in the diagonal. They can be injected along different lines (or even their information collected in completely different ways). Exploring recursions performing different injections is also of interest. In general, other methods can inject the information of the original polynomial into the lifting in a different form that is still easy to recover. This path opens a high variety of possibilities worth exploring in connection with other sequences and triangles of numbers and with known methods of numerical approximation. We can cite here \cite{macnameeii} and how some of these methods could produce a high variety of different ways and strategy to inject the (information of the) univariate polynomials into different forms or concepts of lifting for these objects using the tools provided there.
\end{remark}

This observation also leads us to ask about strategies, tools and methods to measure this accuracy around the diagonal across recursions performed through different stability preserving operators. In particular, we want to compare different operators for their effects on the accuracy of the spectrahedral relaxation (approximation) around the diagonal. Of course, further measures of accuracy of the approximation provided by the spectrahedral relaxation would also be beneficial for that global analysis in terms of the operators used, that is, beyond the globality given merely by analyzing global accuracy of the approximation to RCSs formed by the sequence of approximations generated by spectrahedral relaxation associated to the same operator (recursion). We could go beyond that concept exploring and comparing the effects of different recursions (operators) on the abilities of the spectrahedral relaxation to provide good approximations to the corresponding RCSs.

\begin{remark}[Measuring the accuracy of the spectrahedral relaxation for different operators]
Of course, we measured the effectivity of the spectrahedral relaxation introduced in \cite{main} to produce approximations to the RCSs of the multivariate Eulerian polynomials using their accuracy towards generating bounds for the extreme roots of the polynomials lying in the diagonal. For other operators producing liftings, these measures might have to be refined or expanded. For example, in the case where our polynomials of interest lie far from the diagonal we should measure accuracy through these other lines. But this is not the only approach. As we said before, sometimes the injection of the original polynomials can happen in a completely different way and that would require a new approach towards measuring this accuracy. Additionally, measuring accuracy just through approximations via a line is a good starting point for this task, but it can not be the end of this story since much information is not being used in this way and many structures are therefore being ignored along this simplified path. We do this since there are many complications in going beyond this method. These complications have to be addressed more carefully. The space of just a paper does not allow all the care needed. Thus we decide to proceed in this simplified way because we always a need starting point to justify that some approach work. This is the reason why we worked here in such an anchor for the theory of the measure of the accuracy along the diagonal. However, a comprehensive theory of these measures of accuracy would be something desirable to develop in future works around this topic.
\end{remark}

Finally, it is natural to ask for other related polynomials that \textit{lift} the univariate polynomials in a different sense. This can mean that they lift to other transformations of the roots that allow us to recover the original roots easily or that they do so with the coefficients or other features of the original polynomials. In particular, we search for further and more ample concepts of \textit{lifting} in a broad sense.

\begin{remark}[Liftings as structured transformations]
In this regard, lifting a sequence of univariate polynomials admits many different interpretations depending on the structures we want to maintain and use during the lifting procedure. A deeper look at these different structures and associated techniques would shed more light about the connections between the different features we focus on at each time. For example, here we saw how we can actually reach the same lifting using two completely different approaches: either a combinatorial one or one based on number theory and stability preserving operators and recursion. It turns out that these methods are in fact connected, as can be seen in \cite{visontai2013stable, haglund2012stable}. Studying these relations would show more of these connections. The emergence of these connections establishes thus roads and highways to study deeper mathematical structures involved in these transformations. Additionally, at a more local level, we can witness the emergence of different local structures of our polynomials that can be expanded or generalized towards settings that go beyond the univariateness. For example, in this regard, we can see that multivariateness does not need to be the end of the story of generalization presented here: there are other possible notions of liftings and possibilities of injection of the information using different methods or approaches. We can therefore think further about other (complementary or totally different) notions of lifting that go, along similar lines, through ways and tools allowing us to recover the original information of the polynomials at play in a reasonably easy way.
\end{remark}

This finally lead us to discuss what is the nature expected or hoped of any future improvement using these or different methods, approaches or tools. This is what we briefly glance at in the next section.

\section{The nature of future improvements}\label{nature}

So far, we have obtained bounds for these roots by one side. Having bounds by the other side could help us to understand if a further improvement is possible or if we are already accurate up to some growth term. Therefore, for future improvements, we will need to access strategies allowing us to construct bounds by the side opposite to the one the spectrahedral relaxation gives.

\begin{remark}[Lateral bounds]
The spectrahedral relaxation provides directly outer bounds for the innermost oval of zeros of the RZ polynomials we are dealing with. Here, using the palindromicty of the univariate Eulerian polynomials, we could transform these outer bounds into inner bounds for the outermost oval of zeros of the same polynomials. If we want to determine the asymptotic behaviour of these roots using this method, then we need dual tools providing bounds in the other direction so that we can determine the asymptotic growth through a sandwich of inequalities. For this, we will need to study dual concepts of a relaxation, something that does not appear in the optimization literature and would therefore require a completely novel treatment.
\end{remark}

Finally, these developments could allow use to close the asymptotic gap until we obtain bounds that match the asymptotic expansion of the actual roots up to the exponentially growing terms, that is, up to the terms of the form $a^{n}$ with $a>1$ in the expansion. This would mean that we obtain an almost closed form \textit{at infinity} for the expressions of these extreme roots. For interest in the form of these roots, see \cite{287547} where a different approach is taken in Richard Stanley’s answer.

\begin{remark}[Form at infinity]\label{formatinf}
Managing a good understanding of the asymptotic behaviour of a sequence at infinity provides a nice guidance to roughly understand and work with the quantities at play. In particular, estimating the growth of these roots until the difference with the approximation gets small (maybe even in relative terms) at a considerably fast pace (for whatever measure of this we want to use) gives much more numerical insight about the nature of these quantities than some difficult to write, use, compute and understand algebraic expressions could actually provide. For this reason, we consider legitimately interesting developing a theory about the algebraic asymptotic form at infinity (in a sense determined by the actual accuracy we expect at each moment) of these (especially combinatorial) sequences we are interested in. These forms at infinity could end up shedding new lights about the original objects by offering fresher ideas and further insights about structures closely connected or intimately related to them. Even for these structures that could lie further or be connected in a more loose way, understanding them better and their relations with the original object could cause chain reactions (actual ones) providing deep clues about the behaviour and properties of the original objects that initially interested us. This could eventually serve also as the starting point for researching and exploring further interesting structures emerging from the study, analysis or applications of these methods.
\end{remark}

Here we did not directly investigate the corresponding asymptotic expansions because we were just interested in certifying an (exponentially explosive) improvement produced by the use of the multivariate relaxation. However, the future of our work passes by effectively determining the asymptotic expansion so that we can identify more growth terms of the actual extreme roots until we exhaust the power of our methods.

\begin{remark}[Asymptotic growth terms]
In particular, so far we know that the first term in the asymptotic scale convened in Convention \ref{expsca} and used here is $2^{n+1}$ since we determined this through a sandwich of inequalities in \cite{ale1}. Here we did not directly extract more growth terms due to the fact that we did not determine the actual expansions of the bounds. We just determined the growth term of the sequence of differences between the two different bounds in order to determine who was the winner and \textit{by how much} asymptotically. In the future, we will investigate the actual asymptotic expansions of these roots and bounds in order to try to extract further estimations for deeper terms in the asymptotic expansion of the roots of the univariate Eulerian polynomials in the scale we use here. That deeper knowledge about these further terms of the asymptotic expansion of the extreme roots will set us closer to determine nice forms at infinity of these extreme roots in the sense discussed above in Remark \ref{formatinf}. This is the beginning of our path towards the future development of the concept about the asymptotic algebraic expressions we mentioned as worth studying further there.
\end{remark}

This finishes our analysis of prospective directions for further research deep within this topic. We can now conclude this article.

\section{Conclusion}\label{con}

Here we have constructed Eulerian polynomials as analogues of powers of the linear form $1+x$ using their relation with their corresponding triangle of coefficients. After that, we have seen strategies to lift these univariate polynomials to multivariate ones in a synthetic way, that is, without having to resort to the underlying combinatorics. This synthetic construction showed us a venue for exploring these polynomials from the pure perspective of stability preservers. Moreover, this procedure avoided the appearance of the ghost variables seen in \cite{ale1} and showed, from a different perspective, the connection between these recursions through stability preserving operators and the underlying combinatorial object. This can be explored further in \cite{visontai2013stable, haglund2012stable}.

After briefly exploring this, we computed new bounds for the extreme roots of the univariate Eulerian polynomials using the multivariate ones and the spectrahedral relaxation. We are interested in certifying, through these bounds, how good and tight is the global approximation given by the spectrahedral relaxation measuring the accuracy around the diagonal, as was done in \cite{ale1}. However, there this certification of the power and accuracy of the spectrahedral relaxation around the diagonal was weak since the certified advantage of the multivariate relaxation obtained through the use of the sequence of vectors $\{(y,1,\mathbf{-1})\in\mathbb{R}^{n+1}\}_{n=1}^{\infty}$ fades when $n\to\infty$ because the difference with the univariate bound is dominantly $\frac{1}{2}\left(\frac{3}{4}\right)^{n}$, which goes to $0$ when $n\to\infty$.

However, as the numeric experiments were telling us that the advantage of the multivariate relaxation was bigger for small computable values and we did not devise reasons to believe that this advantage should fade away, we embarked ourselves into the task of finding better sequences of vectors able to provide certifications of the advantage of the multivariate relaxation not vanishing at infinity. More than that, we obtained a certification for this advantage to grow explosively (exponentially) when $n\to\infty$ by using a linearization by a more structured sequence of vectors. We also explored the story of how we built that new and better sequence of vectors.

In particular, we explained how we explored different constructions of vectors numerically until we found a structure that provided, in the low numbers where doing these numerical experiments is possible, the best observed estimation. To do this, we reused the $y$ computed in \cite{ale1}, as it worked as an initial guess to determine good numerical results. All in all, these experiments prompted us to consider the sequence of vectors $$v:=\{(y,(-2^{m-i})_{i=3}^{m},(0,\frac{1}{2}),(1)_{i=1}^{m})\in\mathbb{R}^{n+1}\}_{n=1}^{\infty},$$ which produces a linearization providing a bound that diverges apart from the univariate bound. This gives therefore the kind of certificate of explosive (exponential) asymptotic growth at infinity we wanted.

Thus, using this vector and the $y$ found through an optimization process in \cite{ale1} we computed the corresponding linearization and, proceeding like in \cite{ale1}, we determined the effectivity of such new and more structured sequence of vectors to provide a certificate of growth for the difference between the bounds extracted from the univariate relaxation and the multivariate relaxation. This study, in particular, allowed us to prove and certify that this difference not only does not vanish at infinity, but it grows exponentially when $n\to\infty.$ In fact, we saw that, when measured over even indices $n=2m$, this difference is asymptotically dominantly $\frac{3}{8}\left(\frac{9}{8}\right)^m$, which clearly grows exponentially when $n\to\infty$.

These explorations certify that, through the determination of better, more elaborated sequences of vectors, there is still room for improvement on the certifications of the improvement that multivariateness provides to the performance of the spectrahedral relaxation. Exploring further structures of these vectors and other venues of studying higher structures of lifting to multivariateness that can be effectively combined with the spectrahedral relaxation opens the path to the considerations of further work to do in different directions connected with our findings here and those in \cite{ale1}. With this window to the future, we conclude this article.

\section*{Declaration of competing interest}

The author declares that he has no known competing financial interests or personal relationships that could have appeared to influence the work reported in this paper.

\section*{Acknowledgements}

I thank my doctoral advisor Markus Schweighofer for his support during the development of this research. This work has been generously supported by European Union’s Horizon 2020 research and innovation programme under the Marie Skłodowska Curie Actions, grant agreement 813211 (POEMA).

\end{document}